\documentclass[letter]{amsart}

\usepackage{amssymb}
\usepackage{amsfonts}
\usepackage{amsmath}
\usepackage{amsthm}
\usepackage{mathtools}
\usepackage{graphicx}
\usepackage{mathrsfs}
\usepackage{dsfont}
\usepackage{amscd}
\usepackage{multirow}
\usepackage[all]{xy}
\normalfont
\usepackage[T1]{fontenc}
\usepackage{calligra}
\usepackage{verbatim}
\usepackage{fourier}
\usepackage[usenames,dvipsnames]{color}
\usepackage[colorlinks=true,linkcolor=Blue,citecolor=Violet]{hyperref}
\usepackage{enumerate}
\usepackage{slashed}
\usepackage{nicematrix}
\usepackage{faktor}

\newcommand{\N}{{\mathds{N}}}
\newcommand{\Z}{{\mathds{Z}}}

\newcommand{\R}{{\mathds{R}}}
\newcommand{\C}{{\mathds{C}}}
\newcommand{\T}{{\mathds{T}}}

\newcommand{\D}{{\mathfrak{D}}}
\newcommand{\A}{{\mathfrak{A}}}
\newcommand{\B}{{\mathfrak{B}}}

\newcommand{\Nbar}{\overline{\N}}
\newcommand{\Lip}{{\mathsf{L}}}
\newcommand{\TLip}{{\mathsf{T}}}

\newcommand{\Hilbert}{{\mathscr{H}}}

\newcommand{\dpropinquity}[1]{{\mathsf{\Lambda}^\ast_{#1}}}
\newcommand{\covpropinquity}[1]{{\mathsf{\Lambda}^{\mathrm{cov}}_{#1}}}

\newcommand{\spectralpropinquity}[1]{{\mathsf{\Lambda}^{\mathsf{spec}}_{#1}}}

\newcommand{\Kantorovich}[1]{{\mathsf{mk}_{#1}}}

\newcommand{\KantorovichLength}[1]{{\mathrm{mk}\ell_{#1}}}
\newcommand{\KantorovichDist}[1]{\mathrm{mkD}_{#1}}

\newcommand{\Haus}[1]{{\mathsf{Haus}\!\left[{#1}\right]\,}}

\newcommand{\StateSpace}{{\mathscr{S}}}

\newcommand{\MongeKant}{{Mon\-ge-Kan\-to\-ro\-vich metric}}

\newcommand{\qcms}{quantum compact metric space}
\newcommand{\QCMS}{Quantum Compact Metric Space}

\newcommand{\unit}{1}

\newcommand{\sa}[1]{{\mathfrak{sa}\left({#1}\right)}}

\newcommand{\LargeIso}{{\mathsf{ISO}}}
\newcommand{\smallIso}{{\mathsf{Iso}}}

\newcommand{\dom}[1]{{\operatorname*{dom}\left({#1}\right)}}

\newcommand{\diam}[2]{{\mathrm{diam}\left({#1},{#2}\right)}}
\newcommand{\qdiam}[2]{{\mathrm{qdiam}\left({#1},{#2}\right)}}

\newcommand{\norm}[2]{\left\|{#1}\right\|_{#2}}

\newcommand{\targetsettunnel}[3]{{\mathfrak{t}_{#1}\left({#2}\middle\vert{#3}\right)}}

\newcommand{\CDN}{{\mathsf{DN}}}

\newcommand{\worknote}[1]{}
\newcommand{\opnorm}[3]{{\left|\mkern-1.5mu\left|\mkern-1.5mu\left| {#1} \right|\mkern-1.5mu\right|\mkern-1.5mu\right|_{#3}^{#2}}}

\newcommand{\alg}[1]{{\mathfrak{#1}}}

\newcommand{\Aut}[1]{{\mathrm{Aut}\left({#1}\right)}}

\newcommand{\AdRep}[1]{{\mathrm{Ad}_{#1}}}

\newcommand{\ModState}[1]{\widehat{\StateSpace}}

\newcommand{\closure}[1]{\mathrm{cl}\left({#1}\right)}

\renewcommand{\geq}{\geqslant}
\renewcommand{\leq}{\leqslant}

\newcommand{\Dirac}{{\slashed{D}}}

\newcommand{\solenoid}{\text{\calligra solenoid}}

\theoremstyle{plain}
\newtheorem{theorem}{Theorem}[section]
\newtheorem*{theorem*}{Theorem}

\newtheorem{corollary}[theorem]{Corollary}

\newtheorem{lemma}[theorem]{Lemma}

\newtheorem{theorem-definition}[theorem]{Theorem-Definition}
\newtheorem{hypothesis}[theorem]{Hypothesis}

\theoremstyle{definition}
\newtheorem{definition}[theorem]{Definition}

\newtheorem{notation}[theorem]{Notation}

\newtheorem{convention}[theorem]{Convention}

\theoremstyle{remark}

\newtheorem{example}[theorem]{Example}

\newtheorem{remark}[theorem]{Remark}

\numberwithin{equation}{section}

\usepackage[T1]{fontenc}

\begin{document}

\title[]{ Isometry groups of  inductive limits of metric spectral triples and  Gromov-Hausdorff convergence  }
\author{Jacopo Bassi}
\address{Dipartimento di Matematica, Universit\`a di Roma Tor Vergata, 00133 Rome, Italy.} 
\email{bssjcp01@uniroma2.it}
\author{Roberto Conti}
\address{Dipartimento SBAI, Sapienza Universit\`a di Roma, 00161 Rome, Italy.} 
\email{roberto.conti@sbai.uniroma1.it}
\author{Carla Farsi}
\address{Department of Mathematics, University of Colorado at Boulder, Boulder, CO 80309-0395, USA}
\email{carla.farsi@colorado.edu}

\author{Fr\'{e}d\'{e}ric Latr\'{e}moli\`{e}re}
\address{Department of Mathematics \\ University of Denver \\ Denver CO 80208}
\email{frederic@math.du.edu}
\urladdr{http://www.math.du.edu/\symbol{126}frederic}

\date{\today}

\begin{abstract}
  In this paper we study the groups of isometries and the set of bi-Lipschitz automorphisms of spectral triples  from a metric viewpoint, in the propinquity framework of Latr\'emoli\`ere. In particular we prove that these groups  and sets are compact in the automorphism group of the spectral triple $C^*$-algebra with respect to the Monge-Kantorovich metric, which induces the topology of pointwise convergence. We then prove a necessary and sufficient condition for the convergence of the actions of various groups of isometries, in the sense of the covariant version of the Gromov-Hausdorff propinquity --- a noncommutative analogue of the Gromov-Hausdorff distance --- when working in the context of inductive limits of {\qcms s} and metric spectral triples. We illustrate our work with examples including AF algebras and noncommutative solenoids.
\end{abstract}
\maketitle

\tableofcontents

\section{Introduction}

Spectral triples, originally introduced by Connes in 1985, are a powerful tool that allows the generalization of Riemannian geometry  to the realm of noncommutative spaces and $C^*$-algebras. The classical example is that of the Dirac operator acting on the spinor bundle of a closed manifold. Since then, many spectral triples have been introduced  by several authors for many different types of (unital) $C^*$-algebras. The importance of spectral triples lies in their well-established power in generalizing   differential and spectral geometry to various new situations, from  spaces of leaves of foliations, to geometries on quantum tori, quantum spheres, and other quantum spaces, with rich applications in   mathematical physics. In particular, the class of spectral triples obtained as inductive limits of other spectral triples have attracted a lot of attention in recent years. For instance, spectral triples over AF algebras were first introduced by Christensen-Ivan in \cite{CI}; see \cite{CIL}, \cite{CIS}, \cite{LaSa}, \cite{AguLat}, \cite{BaDa} (among many other references) for further investigations of these spectral triples and generalizations. In a quite different direction, spectral triples on inductive limits of twisted group C*-algebras are introduced in \cite{FaLaPa} by the last two authors, together with J. Packer, with examples including the Bunce-Deddens algebras and the noncommutative solenoids (inductive limits of quantum tori) --- $C^\ast$-algebras which are not AF. These constructions  were inspired by the work in  \cite{FlGh} and \cite{FaLaLaPa}. In this paper, we study the geometry of certain natural groups of isometries associated with {\qcms s} and with metric spectral triples, with a particular focus on inductive limits of spectral triples. 

\medskip

Spectral triples induce an extended pseudo-metric on the state space of their underlying (unital) C*-algebras, called the {\it Connes metric}. Of prime interest in noncommutative metric geometry are the spectral triples whose Connes' pseudo-metric is actually a metric that induces the weak* topology; hence the name {\it metric spectral triples}. Metric spectral triples in turn provide examples of {\qcms s}, which are noncommutative analogues of the algebras of Lipschitz functions over a compact metric space. A {\qcms} consists of a unital C*-algebra endowed with the noncommutative analogue of a Lipschitz seminorm, whose main property is that it induces by duality a distance on the state space of the C*-algebra, which metrizes the weak* topology, in a manner akin to the Kantorovich metric. This metric is referred to as the \emph{\MongeKant} of the {\qcms}. 

Inspired by Connes' ideas, in  2004 Rieffel formalized in \cite{Rieffel00} the idea of convergence of {\qcms s} (with a more general definition involving order unit spaces rather than C*-algebras) by introducing the quantum Gromov-Hausdorff distance, a form of the Gromov-Hausdorff distance between the state spaces of {\qcms s} endowed with the \MongeKant.  Latr\'emoli\`ere extended  Rieffel's ideas, incorporating the $C^*$-algebra multiplicative structure by requiring that the analogue of a Lipschitz seminorm for {\qcms s} satisfy a form of the Leibniz inequality. With that crucial addition, Latr\'emoli\`ere was able to prove that distance zero between two {\qcms s} is equivalent to the underlying C*-algebras being $\ast$-isomorphic, together with their state spaces being isometric. This new  distance was called {\it (dual) propinquity} \cite{Latremoliere13b}. In a sequence of papers, Latr\'emoli\`ere has vastly extended the domain of the propinquity to include, besides  {\qcms s}, such structures as $C^*$-bundles \cite{Latremoliere16c,Latremoliere16d}, $C^*$-correspondences \cite{Latremoliere16d}, Lipschitz dynamical systems (which involve actions of groups on \qcms s) \cite{Latremoliere18b,Latremoliere18c}, and metric spectral triples \cite{Latremoliere18g}.

In all of these cases, distance zero is equivalent to the given structures being isomorphic in a natural sense. Moreover, these metrics are complete on large classes. In particular, the \emph{spectral propinquity}, which is defined for the class of metric spectral triples  \cite{Latremoliere18g}, is zero between two metric spectral triples if and only if they are unitarily equivalent. Moreover this metric dominates the propinquity, and thus implies metric convergence of the quantum compact metric spaces associated to  metric spectral triples. This latest extension  opened up the  possibility of  using the tools of metric geometry on the class of metric spectral triples.

An interesting question arises as to what isometries might be in the context of noncommutative metric geometry, and in particular (metric) spectral triples. For this, Latr\'emoli\`ere introduced the notions of bi-Lipschitz morphisms and {\it full quantum isometry} for quantum compact metric spaces; see \cite{Latremoliere16b}. For  spectral triples, the notion of isometry is quite significant and in the classical example   of the Dirac operator acting on the spinor bundle of a closed manifold, isometries should be the standard Riemannian metric isometries. So the questions of what  isometries for spectral triples need to be, and what is their relation  to the metric structures defined on spectral triples by Latr\'emoli\`ere naturally arises in this context. 

\medskip
Various notions of isometries of spectral triples  have appeared in the literature in the setting either of classical groups  or of quantum groups.  In this work we will focus on the  two classical  isometry   groups, called $\smallIso$ and $ \LargeIso$, see Definitions \eqref{small-iso-def}, \eqref{Large-iso-def}.  For completeness,  we note that there also exists a notion of quantum group of isometries see e.g. \cite{GoBh} and the references, as well as \cite{BGS-2011} for 0-dimensional manifolds quantum case.
Both  $\smallIso$ and $ \LargeIso$  are  subgroups of the $\ast$-automorphism group of the underlying $C^*$-algebra: the group $\LargeIso$ is the isometry group of the associated Monge-Kantorovich metric structure, while $\smallIso$, introduced by Park in \cite{Efton95}, is more restricted and requires a unitary that implements the given $C^*$-algebra automorphism, which  also commutes with the Dirac; in general,  $\smallIso$ is always a subgroup of $\LargeIso$.  In addition, our emphasis will be on spectral triples obtained as inductive limits, and in this setting we  use the covariant propinquity for actions of proper groups on {\qcms s} to prove that certain   groups of isometries   converge to  projective limits of groups of isometries. This work requires us to first prove that the group of isometries of a {\qcms} is compact.  For the Riemannian  manifold case,  noncommutative tori and other examples, $\smallIso$ was studied by  Park in  \cite{Efton95, Park2}, for spectral triples over twisted reduced group $C^*$-algebras associated to a length function as defined in \cite{Connes89} by Long-Wu in \cite{LoWu},  for Goffeng-Mesland  \cite{GM} spectral triples  on  Cuntz algebras by Conti-Rossi  in \cite{CoRo},  and for  the Christensen-Ivan \cite{CI} spectral triples of AF-algebras by Bassi-Conti in \cite{Conti21}. Conti-Farsi consider both $\smallIso$ and $\LargeIso$ for Kellendonk-Savinien spectral triples in \cite{CoFa}.

In this paper, we start with  the  more general framework of bi-Lipschitz automorphisms, and  then
consider the $\LargeIso$ and $\smallIso$ groups of isometries of spectral triples as subsets of the bi-Lipschitz automorphisms.
The first main results of this paper are  that  the set of the $K$-bi-Lipschitz automorphisms, and the 	groups $\smallIso$ and $\LargeIso$   are always closed (and hence compact) in the Monge-Kantorovich metric  $\KantorovichDist{\Lip}$ defined on the automorphism group $\Aut{\A}$  of $\A$.

\begin{theorem*} (Theorem \eqref{compact-Iso-thm}, Corollary \eqref{cor:ISO-compact}, and Theorem \eqref{compact-iso-group-thm}) Let  $(\A,\Lip)$ be a {\qcms}. Let $\Aut{\A}$ be the group of $\ast$-automorphisms of $\A$, whose topology of pointwise convergence is metrized by the distance
  \begin{equation*}
    \forall\alpha,\beta\in\Aut{\A} \quad \KantorovichDist{\Lip}(\alpha,\beta) \coloneqq \sup\left\{ \norm{\alpha(a)-\beta(a)}{\A} : a\in\dom{\Lip}, \Lip(a) \leq 1 \right\} \text.
  \end{equation*}

  The following subsets and subgroups of $\Big(\Aut{\A}, \KantorovichDist{\Lip}\Big) $ are compact. 
	\begin{enumerate}
		\item The set of bi-Lipschitz automorphisms $\Aut{\A,\Lip,K}$, with $K \geq 1$:
		\begin{equation*}
		\Aut{\A,\Lip,K} \coloneqq \left\{ \alpha \in \Aut{\A} : \text{$\alpha$ is $K$-bi-Lipschitz} \right\}.
		\end{equation*}	
		\item The subgroup $\LargeIso(\A,\Lip)\coloneqq \{ \alpha \in \Aut{\A}:\ \text{is a  full quantum isometry.}\}$.
		\item For a metric spectral triple,   $(\A,\Hilbert,\Dirac)$ the subgroup $\smallIso(\A,\Hilbert,\Dirac)$ defined by:
		\begin{multline*}
		\smallIso(\A,\Hilbert,\Dirac) \coloneqq \\ \left\{ \alpha \in \mathrm{Aut}(\A) : \exists U \text{ unitary on $\Hilbert$ s.t. \ } \AdRep{U} = \alpha, U\dom{\Dirac}=\dom{\Dirac}, [\Dirac,U] = 0 \right\} \text.
		\end{multline*}
	\end{enumerate}
\end{theorem*}

By \cite[Theorem 2.22]{FaLaPa}, given an inductive sequence of {\qcms s}, whose limit is a {\qcms}, we have that the existence of  a special type of $\ast$-automorphism on the inductive limit which is well behaved with respect to the quantum metrics, called a \emph{bridge builder}, is equivalent, under a very natural assumption, to the convergence of the given sequence of {\qcms s} for the propinquity. Examples of such construction were noncommutative solenoids and bounded Bunce-Deddens algebras, see \cite{FaLaPa}.

Now, given that (metric) spectral triples were thus  enriched with their isometry groups structures, a natural question   to ask was if  one could extend the original   bridge builder construction of \cite{FaLaPa}  in such a way it insured the   convergence of the associated isometry groups. This is what we are doing in this work.
Covariant bridge builders are indeed full quantum isometries of the inductive limit space that are covariant with respect to the actions, and behave well with respect to the approximations in the limit.
See  Definition \eqref{cov-bridge-def} for details.

With this, we are now ready to state  the second main results of this paper.

\begin{theorem*}[Theorems \eqref{covariant-bridge-builder-thm} and \eqref{covariant-bridge-builder-converse-thm}] For each $n\in\Nbar \coloneqq \N\cup\{\infty\}$, let $(\A_n,\Lip_n)$ be a {\qcms}, such that $\A_\infty = \closure{\bigcup_{n\in\N} \A_n}$, where $(\A_n)_{n\in\N}$ is an increasing (for $\subseteq$) sequence of C*-subalgebras of $\A_\infty$, with the unit of $\A_\infty$ in $\A_0$.
	Assume that 	
	$(\A_n,\Lip_n)_{n\in\Nbar}$ are  {\qcms s},  and for each $n\in\Nbar$, let $\alpha_n$ be an action of a group $G_n$ on $(\A_n,\Lip_n)$ by full quantum isometries, satisfying Hypothesis (\ref{cov-bridge-hyp}).
	
	If there exists a covariant bridge builder for this data, then
	\begin{equation}\label{cv-eq}
	\lim_{n\rightarrow\infty} \covpropinquity{}((\A_n,\Lip_n,\alpha_n,G_n),(\A_\infty,\Lip_\infty,\alpha_\infty,G_\infty)) = 0 \text.
	\end{equation}
Conversely, if we assume that there exists $M > 0$ such that, for all $n\in\N$, we have $\frac{1}{M}\Lip_n\leq\Lip_\infty\leq M\Lip_n$ on $\dom{\Lip_n}$, and if Equation (\ref{cv-eq}) holds, then there exists a covariant bridge builder.
\end{theorem*}

In this paper, the main application of Theorem \eqref{covariant-bridge-builder-thm} is the following convergence result:

\begin{theorem*}(Theorem \eqref{thm:iso-conv})
	For each $n\in\Nbar \coloneqq \N\cup\{\infty\}$, let $(\A_n,\Lip_n)$ be a {\qcms} such that $\A_\infty = \closure{\bigcup_{n\in\N} \A_n}$, where $(\A_n)_{n\in\N}$ is an increasing (for $\subseteq$) sequence of C*-subalgebras of $\A_\infty$, with the unit of $\A_\infty$ in $\A_0$. Let $\pi$ be a full quantum isometry which is also a bridge builder. 
	The group $\LargeIso(A_\infty,\Lip_\infty|\pi)$ defined by:
	\begin{equation*}
	\LargeIso(\A_\infty,\Lip_\infty|\pi) \coloneqq \left\{ \alpha\in \LargeIso(\A_\infty,\Lip_\infty) : \pi\circ\alpha=\alpha\circ\pi\text{ and }\forall n \in \N: \quad \alpha_{|\A_n} \in \LargeIso(\A_n,\Lip_n) \right\} \text.
	\end{equation*}
	is a compact subgroup of the $\ast$-automorphism $\Aut{\A_\infty}$.
	
	Furthermore, if $G_\infty$ is a closed subgroup of $\LargeIso(\A_\infty,\Lip_\infty|\pi)$ and if, for all $n\in\N$, we set $G_n \coloneqq \left\{ \alpha_{|\A_n} \in \Aut{\A_n} : \alpha\in G_\infty \right\}$, then $G_n$ is a closed group, and
	\begin{equation*}
	\lim_{n\rightarrow\infty} \covpropinquity{}((\A_n,\Lip_n,G_n),(\A_\infty,\Lip_\infty,G_\infty)) = 0 \text,
	\end{equation*}
	where $\covpropinquity{}$ is the covariant propinquity (see \cite{Latremoliere18b}).
\end{theorem*}

In Section\eqref{sec:applications} we provide some applications of the results of Section  \eqref{sec:prop-isom-gps}. In particular, when the identity is a covariant bridge builder (see Corollaries \eqref{id-cov-cor} and \eqref{id-cov-small-cor}), the main problem is reduced to  the extension of quantum isometries from a subalgebra to an algebra. Even though such extensions are not always possible, we specialize to the case of inner automorphisms, which can always be extended. This viewpoint allows us to show in Theorem \eqref{conv-AF} that in the case of inductive limits of metric spectral triples associated to AF algebras \cite{CI}, the  ISO group of the limit  is the limit of the ISO groups of the factors. In the case of UHF algebras, more is possible, and the limit result extends to the groups Iso as well, in accordance with \cite{Conti21}. That is,
if $\A_\infty$ is a UHF algebra, with $\A_n = \otimes_{j=0}^n \alg{M}(\delta_j)$ where $\delta_0=1$, $(\delta_n)_{n\in\N}$ a sequence of natural numbers, and $\alg{M}(d)$ is the C*-algebra of $d\times d$ matrices, then for all $n\in\N$, define:
\begin{equation*}
\smallIso(\A_n,L^2(\A_n,\varphi),(\Dirac_\lambda)_{|\A_n}) = \left\{ \otimes_{j=0}^n \AdRep{u_j} : \forall j \in \{0,\ldots,d\} \quad u_j\in\alg{M}(\delta_j), u_j^\ast=u_j^{-1} \right\} \text,
\end{equation*}
and
\begin{equation*}
\smallIso(\A_\infty,L^2(\A_\infty,\varphi),\Dirac_\lambda) = \left\{ \alpha \in \Aut{\A_\infty} : \alpha_{|\A_n} \in \smallIso(\A_n,L^2(\A_n,\varphi),\Dirac_\lambda) \right\} \text.
\end{equation*}
Then:
\begin{equation*}
\lim_{n\rightarrow\infty} \covpropinquity{}((\A_n,\Lip_n,\smallIso(\A_n,L^2(\A_n,\varphi),\Dirac_n), (\A_\infty,\Lip_\infty,\smallIso(\A_\infty,L^2(\A_\infty,\varphi),\Dirac_\lambda))) = 0 \text.
\end{equation*}

In Section 4 we also consider spectral triples on noncommutative solenoids and Bunce-Deddens algebras, in which case we obtain convergence for the dual action, see Section \eqref{sec:dual-actions}.

\subsection*{Acknowledgments}

J.B. was supported by the grant \lq \lq Algebre di Operatori e Teoria Quantistica dei Campi", CUP: E83C18000100006; he also acknowledges the University of Rome \lq \lq Tor Vergata" funding OAQM, CUP: E83C22001800005.
R.C. was partially supported by Sapienza, Progetti di Ricerca Scientifica 2019 and  2020.
This work was also partially supported by C. F.'s Simons Foundation Collaboration grants \#523991. C.F. also thanks  the Dept. SBAI, Sapienza  Universit\`a di Roma, for their kind hospitality.


\section{The Isometry Groups of a Metric Spectral Triple}\label{sec:isom-metric-sp-tr}

There are at least two interesting groups of automorphisms associated with a metric spectral triple which can be referred to as groups of isometries. As metric spectral triples give rise to {\qcms s}, the larger of these two groups consists of all full quantum isometric automorphisms for the quantum metric. On the other hand, one may only restrict one's attention to those isometries which are directly compatible with the spectral triple itself, i.e. automorphisms which are spatially implemented by unitaries which commute with the Dirac operator. In this section, we study the basic topology of both of these groups. As our goal is to study the covariant convergence of certain closed subgroups of isometries on {\qcms s}, it is important to check that both our groups of isometries are compact in the natural topology, which for our purpose, is the topology of pointwise convergence, which is metrized by a form of the {\MongeKant} using the underlying quantum metric \cite{Latremoliere16b}. These results are established in Theorem \eqref{compact-Iso-thm}, Corollary \eqref{cor:ISO-compact}, and Theorem \eqref{compact-iso-group-thm}. In passing, we will prove a compactness result for sets of bi-Lipschitz morphisms.

\subsection{The Isometry Group of a \QCMS}\label{sec:isom-qcms}

A {\qcms} is an analogue of the algebra of Lipschitz functions over a compact metric space, designed to provide a quantum version of compact metric spaces in the same spirit as unital C*-algebras generalize compact Hausdorff spaces and spectral triples generalize Riemannian manifolds. 

\begin{notation}
	For any unital $C^\ast$-algebra $\A$, we denote the space $\{a\in\A:a=a^\ast\}$ of self-adjoint elements of $\A$ by $\sa{\A}$. Moreover, the state space of $\A$ is denoted by $\StateSpace(\A)$. The norm of $\A$ is denoted by $\norm{\cdot}{\A}$. 
\end{notation}

\begin{definition}[{\cite{Connes89,Rieffel98a,Rieffel99,Latremoliere13,Latremoliere15}}]\label{qcms-def}
  A \emph{\qcms} $(\A,\Lip)$ is an ordered pair where $\A$ is a unital C*-algebra and $\Lip$ is a seminorm defined on a dense Jordan-Lie subalgebra $\dom{\Lip}$ of $\sa{\A}$ such that:
  \begin{enumerate}
  \item $\{ a \in \sa{\A} : \Lip(a) = 0 \} = \R\unit_\A$,
  \item The {\MongeKant} $\Kantorovich{\Lip}$, defined on the state space $\StateSpace(\A)$ of $\A$ by, for all $\varphi,\psi \in \StateSpace(\A)$:
    \begin{equation*}
      \Kantorovich{\Lip}(\varphi,\psi) \coloneqq \sup\left\{ \left|\varphi(a)-\psi(a)\right| : a\in\dom{\Lip},\Lip(a)\leq 1 \right\}
    \end{equation*}
    metrizes the weak* topology when restricted to $\StateSpace(\A)$;
  \item there exists $\Omega\geq 1$ and $\Omega' \geq 0$ such that $\Lip$ has the $(\Omega,\Omega')$-quasi-Leibniz property, i.e. for all $a,b \in \dom{\Lip}$:
    \begin{equation*}
      \max\left\{\Lip\left(\frac{ab+ba}{2}\right),\Lip\left(\frac{ab-ba}{2i}\right)\right\} \leq \Omega(\norm{a}{\A}\Lip(b) + \Lip(a) \norm{b}{\A}) + \Omega' \Lip(a)\Lip(b)\text;
    \end{equation*}
  \item $\left\{ a \in \dom{\Lip} : \Lip(a) \leq 1 \right\}$ is closed in $\sa{\A}$.
  \end{enumerate}
\end{definition}

By \cite[Theorem 1.9]{Rieffel98a}, we note that Condition (2) and Condition (3) are equivalent to the single condition that there exists a state $\mu\in \StateSpace(\A)$ of $\A$ such that the set $\{ a\in\dom{\Lip} : \Lip(a)\leq 1,\mu(a)=0\}$ is compact in $\sa{\A}$.

We shall use the following  convenient convention.
\begin{convention}\label{dom-convention}
  If $L : D\rightarrow \R$ is a seminorm over a dense subspace $D$ of some Banach space $E$, then we set $\dom{L} = D$ and $L(z) = \infty$ for all $z\in E\setminus D$. We use the standard conventions that, for $x \in \R$: $\infty > x$, $\infty+x=x+\infty=\infty$ for all $x\geq 0$, and $x\infty=\infty$ for $x>0$, while $0\infty=0$.
\end{convention}

Since the state spaces of {\qcms s} are indeed compact metric spaces, they have finite diameter; this observation plays a useful role in our work.
\begin{notation}
  If $(X,d)$ is a metric space, then we define the {\it diameter of $X$} by  $\diam{X}{d}\coloneqq\sup\{ d(x,y):x,y \in X \}$. If $B\subseteq E$ is a subset of a normed vector space $E$, then we define $\diam{B}{E}$ or $\diam{B}{\norm{\cdot}{E}}$ to be the diameter of $B$ for the metric induced by the norm $\norm{\cdot}{E}$ of $E$.
  If $(\A,\Lip)$ is a {\qcms}, then we set
  \begin{equation*}
    \qdiam{\A}{\Lip} \coloneqq \diam{\StateSpace(\A)}{\Kantorovich{\Lip}}\text.
  \end{equation*}
\end{notation}

Quantum compact metric spaces are objects of several categories, where the arrows are $\ast$-morphisms with additional compatibility with the quantum metric data.

\begin{definition}[{\cite{Rieffel99,Rieffel00,Latremoliere16b}}]\label{Lipschitz-morphism-def}
  Let $(\A,\Lip_\A)$ and $(\B,\Lip_\B)$ be two {\qcms s}.
  
  A \emph{Lipschitz morphism} $\pi : (\A,\Lip_\A) \rightarrow (\B,\Lip_\B)$ is a $\ast$-morphism from $\A$ to $\B$ such that $\pi(\dom{\Lip_\A})\subseteq\dom{\Lip_\B}$.

  A \emph{quantum isometry} $\pi : (\A,\Lip_\A)\rightarrow(\B,\Lip_\B)$ is a Lipschitz morphism such that $\pi$ is surjective, and
  \begin{equation*}
    \forall b\in\dom{\Lip_\B} \quad \Lip_\B(b) = \inf\left\{ \Lip_\A(a) : \pi(a) = b \right\} \text. 
  \end{equation*}

  A \emph{full quantum isometry} $\pi : (\A,\Lip_\A) \rightarrow (\B,\Lip_\B)$ is a quantum isometry such that $\pi$ is a $\ast$-isomorphism, and $\pi^{-1} : (\B,\Lip_\B)\rightarrow(\A,\Lip_\A)$ is also a quantum isometry.
\end{definition}

\begin{remark}
  With our Convention (\ref{dom-convention}), we note that, for instance, $\pi:(\A,\Lip_\A)\rightarrow(\B,\Lip_\B)$ is a quantum isometry if, and only if, $\pi$ is a surjective $\ast$-morphism from $\A$ to $\B$ such that
  \begin{equation*}
    \forall b \in \sa{\B} \quad \Lip_\B(b) = \inf\left\{ \Lip_\A(a) : \pi(a) = b \right\} \text.
  \end{equation*}
\end{remark}

One can construct a  category in which the objects are {\qcms s} and the Lipschitz morphisms are the arrows, since composition of Lipschitz morphisms is a  Lipschitz morphism, and the identity is  a Lipschitz morphism, too. As shown in \cite[Proposition 3.7]{Rieffel00}, the composition of quantum isometries is also a quantum isometry (in our context, our isometries are $\ast$-morphisms, rather than unital positive linear maps as in \cite{Rieffel00}, but the argument about their composition carries).

As shown in \cite[Corollary 2.3]{Latremoliere16b}, a $\ast$-morphism $\pi : \A\rightarrow\B$ is a Lipschitz morphism if, and only if, there exists $K > 0$ such that for all $a\in\sa{\A}$, we have $\Lip_\B\circ\pi(a) \leq K \Lip_\A(a)$. Moreover, $\pi$ is a full quantum isometry if, and only if, $\pi(\dom{\Lip_\A}) = \dom{\Lip_\B}$ and $\Lip_\B\circ\pi(a) = \Lip_\A(a)$ for all $a\in\dom{\Lip_\A}$. While full quantum isometries are our focus and are the proper notion of isomorphism for our purpose, we also will work with the weaker notion of a bi-Lipschitz isomorphism, defined as follows.
\begin{definition}
  For $K\geq 1$, a \emph{$K$-bi-Lipschitz isomorphism} $\alpha : (\A,\Lip_\A)\rightarrow(\B,\Lip_\B)$ between two {\qcms s} $(\A,\Lip_\A)$ and $(\B,\Lip_\B)$, is a $\ast$-isomorphism from $\A$ to $\B$ such that $\alpha(\dom{\Lip_\A}) = \dom{\Lip_\B}$ and, for all $a \in \dom{\Lip_\A}$,
  \begin{equation*}
    \frac{1}{K} \Lip_\A(a) \leq \Lip_\B(\alpha(a)) \leq K\, \Lip_\A (a) \text.
  \end{equation*}
\end{definition}

In particular, a $\ast$-isomorphism is a full quantum isometry if, and only if, it is $1$-bi-Lipschitz. We also observe that if $\alpha$ is a $K$-bi-Lipschitz isomorphism, then so is $\alpha^{-1}$. Indeed, since $\alpha(\dom{\Lip_\A})=\dom{\Lip_\B}$, we conclude $\alpha^{-1}(\dom{\Lip_\B}) = \dom{\Lip_\A}$. For all $b \in \dom{\Lip_\B}$, we then have:
\begin{equation*}
  \frac{1}{K}\Lip_\A(\alpha^{-1}(b)) \leq \underbracket[1pt]{\Lip_\B(\alpha(\alpha^{-1}(b)))}_{=\Lip_\B(b)}  \leq K \Lip_\A(\alpha^{-1}(b))
\end{equation*}
So $\alpha^{-1}$ is also $K$-bi-Lipschitz.
Moreover,  if $\alpha: ({\mathfrak A},L_{\mathfrak A}) \to ({\mathfrak B},L_{\mathfrak B})$ is K bi-Lipschitz and $\beta: ({\mathfrak B},L_{\mathfrak B}) \to ({\mathfrak C},L_{\mathfrak C})$ is H-bi-Lipschitz then $\beta \circ \alpha$ is KH-bi-Lipschitz.

There exists a natural Monge-Kantorovich metric on the group $\Aut{\A}$ of $\ast$-automorphisms of a unital C*-algebra $\A$ endowed with a Lipschitz seminorm, the definition of which we now recall from {\cite[Theorem 5.1]{Latremoliere16b}}.
\begin{definition}[{\cite{Latremoliere16b}}]\label{Kantorovich-morphism-def}
  Let $(\A,\Lip)$ be a {\qcms}. For any two $\ast$-automorphisms $\alpha$ and $\beta$ of $\A$, we define the {\it Monge-Kantorovich metric}:
  \begin{equation}\label{eq:def-MK-for-autos}
    \KantorovichDist{\Lip}(\alpha,\beta) \coloneqq \sup\left\{ \norm{\alpha(a)-\beta(a)}{\A} : a\in\dom{\Lip}, \Lip(a) \leq 1 \right\} \text.
  \end{equation}
\end{definition}

The  following basic topological properties of the Monge-Kantorovich  metric on the group of $\ast$-automorphisms were established in \cite{Latremoliere16b}.

\begin{theorem}[{\cite[Theorem 5.2]{Latremoliere16b}}]\label{KantorovichDist-thm}
  If $(\A,\Lip)$ is a {\qcms}, then $\KantorovichDist{\Lip}$ is a metric on the group $\Aut{\A}$ of $\ast$-automorphisms of $\A$, which metrizes the topology of strong convergence, i.e. for a sequence $(\alpha_n)_{n\in\N}$ in $\Aut{\A}$ and $\alpha \in \mathrm{Aut}(\A)$, we have  $\lim_{n\rightarrow\infty}\KantorovichDist{\Lip}(\alpha_n,\alpha) = 0$ if, and only if $\lim_{n\rightarrow\infty} \norm{\alpha_n(a)-\alpha(a)}{\A} = 0$ for all $a\in\A$. Moreover, if $C>0$, then the set $\left\{\alpha\in\Aut{\A} : \Lip\circ\alpha\leq C \Lip\right\}$ is totally bounded for $\KantorovichDist{\Lip}$.
\end{theorem}

The distance in Definition (\ref{Kantorovich-morphism-def}) is induced by a length function on the group of automorphisms, and the inverse operation of automorphisms is bi-Lipschitz on the set of bi-Lipschitz morphisms, for any $K \geq 1$.

\begin{lemma}\label{Kantorovich-Length-lemma}
  Let $(\A,\Lip)$ be a {\qcms}. Define the Monge-Kantorovich length $\KantorovichLength{\Lip}  $ on $\Aut{\A}$ by, for all $\alpha \in  \Aut{\A}$:
  \begin{equation}\label{eq:def-of-mkell}
    \KantorovichLength{\Lip} (\alpha) \coloneqq\sup\left\{ \norm{\alpha(a)-a}{\A} : a\in\dom{\Lip},\Lip(a)\leq 1\right\}.
  \end{equation}
  Then 
  \begin{enumerate}
 \item  $\KantorovichLength{\Lip}$ is a length function on $\Aut{\A}$ such that, for all $\alpha,\beta\in \mathrm{Aut}(\A)$:
  \begin{equation*}
      \KantorovichLength{\Lip}(\beta^{-1}\circ\alpha) =\KantorovichDist{\Lip}(\alpha,\beta) \text.
  \end{equation*}
 \item  For all $\alpha,\beta\in \mathrm{Aut}(\A)$, if $\alpha$ is $K$-bi-Lipschitz for some $K\geq 1$, then we have:
  \begin{equation*}
    \frac{1}{K} \KantorovichDist{\Lip}(\alpha,\beta)\leq \KantorovichDist{\Lip}(\alpha^{-1},\beta^{-1}) \leq K \KantorovichDist{\Lip}(\alpha,\beta) \text.
  \end{equation*}
 \item  In particular, if $\beta \in \Aut{\A}$ and $\alpha$ is a full quantum isometry, we have \begin{equation*}
 \KantorovichLength{\Lip}(\alpha^{-1} \circ \beta \circ \alpha) = \KantorovichLength{\Lip}(\beta)\text.
 \end{equation*}
 Consequently,
  \begin{equation*}
    \KantorovichDist{\Lip}(\alpha^{-1},\beta^{-1}) = \KantorovichDist{\Lip}(\alpha,\beta) \text.
  \end{equation*}
  \end{enumerate}
\end{lemma}

\begin{proof}
	We now provide a proof for each item in our lemma.
 \begin{enumerate}
 	\item  Since $\ast$-automorphisms of $\A$ are isometries for the norm of $\A$, we observe that for all $a\in\A$, we have:
  \begin{equation*}
    \norm{\alpha(a)-\beta(a)}{\A} = \norm{\beta(\beta^{-1}\circ\alpha(a)-a)}{\A} = \norm{\beta^{-1}\circ\alpha(a)-a}{\A} \text,
  \end{equation*}
  so $\KantorovichDist{\Lip}(\alpha,\beta) = \KantorovichLength{\Lip}(\beta^{-1}\circ\alpha)$.
In particular, choosing $\alpha$ to be the identity, we see that
  \begin{equation*}
    \norm{\beta(a)-a}{\A} = \norm{a-\beta^{-1}(a)}{\A} \text,
  \end{equation*}
  i.e., $\KantorovichLength{\Lip}(\beta^{-1}) = \KantorovichLength{\Lip}(\beta)$.
  
 If $\KantorovichLength{\Lip}(\alpha) = 0$ then $\alpha(a) = a$ for all $a\in\dom{\Lip}$ with $\Lip(a)\leq 1$; therefore $\alpha(a) = a$ for all $a\in\A$ as $\alpha$ is continuous and linear, and since $\{a\in\dom{\Lip}:\Lip(a)\leq 1\}$ is total in $\sa{\A}$.
Moreover, for all $\alpha,\beta\in\Aut{\A}$, we have
  \begin{equation*}
    \norm{\alpha\circ\beta(a)-a}{\A}\leq\norm{\alpha(\beta(a))-\alpha(a)}{\A} + \norm{\alpha(a)-a}{\A} = \norm{\beta(a)-a}{\A} + \norm{\alpha(a)-a}{\A}\text,
  \end{equation*}
  which implies that $\KantorovichLength{\Lip}(\alpha\circ\beta)\leq\KantorovichLength{\Lip}(\alpha)+\KantorovichLength{\Lip}(\beta)$.

  Since $\KantorovichLength{\Lip}$ is zero at the identity, we have established that $\KantorovichLength{\Lip}$ is indeed a length function on $\Aut{\A}$.

 \item  Let $\alpha\in \Aut{\A}$ be $K$-bi-Lipschitz for some $K > 0$; this implies by Definition (\ref{Lipschitz-morphism-def}) that  $\alpha^{-1}$ is also $K$-bi-Lipschitz.  Moreover,
  \begin{equation*}
    \left\{ b \in \dom{\Lip} : \Lip(b)\leq 1 \right\} \subseteq \left\{ \alpha(a) : a \in \dom{\Lip}, \Lip(a) \leq K \right\} 
  \end{equation*}
  and
   \begin{equation*}
  \left\{ a \in \dom{\Lip} : \Lip(a)\leq 1 \right\} \subseteq \left\{ \alpha^{-1}(b) : b \in \dom{\Lip}, \Lip(b) \leq K \right\} \text.
  \end{equation*}

  Thus,
  \begin{align*}
    \KantorovichDist{\Lip}(\alpha^{-1},\beta^{-1})
    &=\sup\left\{ |\alpha^{-1}(b)-\beta^{-1}(b)| : b\in\dom{\Lip},\Lip(b) \leq 1 \right\} \\
    &\leq\sup\left\{ |\alpha^{-1}(\alpha(a))-\beta^{-1}(\alpha(a))|: a\in\dom{\Lip}, \Lip(a)\leq K \right\}\\
    &=\sup\left\{ |a-\beta^{-1}(\alpha(a))|: a\in\dom{\Lip}, \Lip(a)\leq K \right\}\\
    &=K\, \KantorovichLength{\Lip}(\beta^{-1}\circ\alpha) \\
    &=K\, \KantorovichDist{\Lip}(\alpha,\beta) \text.
  \end{align*}
  Thus, by symmetry, we obtain:
  \begin{equation*}
    \frac{1}{K}\KantorovichDist{\Lip}(\alpha,\beta) \leq \KantorovichDist{\Lip}(\alpha^{-1},\beta^{-1}) \leq K\, \KantorovichDist{\Lip}(\alpha,\beta) \text.
  \end{equation*}

\item Let  $\alpha$ be a full quantum isometry of $(\A,\Lip)$; then:
  \begin{equation*}
    \left\{a\in\dom{\Lip} : \Lip(a)\leq 1 \right\} = \alpha\left(\left\{ a \in \dom{\Lip} : \Lip(a)\leq 1\right\} \right)\text.
  \end{equation*}
  
  Therefore,
  \begin{align*}
    \KantorovichLength{\Lip}(\alpha^{-1}\circ\beta\circ\alpha)
    &=\sup\left\{ \norm{\alpha^{-1}\circ\beta\circ\alpha(a)-a}{\A} : \Lip(a)\leq 1 \right\} \\
    &=\sup\left\{ \norm{\alpha^{-1}(\beta(\alpha(a))-\alpha(a))}{\A} : \Lip(a)\leq 1 \right\} \\
    &=\sup\left\{ \norm{\beta(\alpha(a))-\alpha(a)}{\A} : \Lip(a)\leq 1 \right\} \\
    &=\sup\left\{ \norm{\beta(a)-a}{\A} : \Lip(a)\leq 1 \right\} \\
    &=\KantorovichLength{\Lip}(\beta) \text.
  \end{align*}

  In particular,
  \begin{align*}
    \KantorovichDist{\Lip}(\alpha^{-1},\beta^{-1})
    &= \KantorovichLength{\Lip}(\beta\circ\alpha^{-1}) \\
    &= \KantorovichLength{\Lip}(\alpha^{-1}\circ(\beta\circ\alpha^{-1})\circ\alpha) \\
    &= \KantorovichLength{\Lip}(\alpha^{-1}\circ\beta) \\
    &= \KantorovichDist{\Lip}(\beta,\alpha) \text.
  \end{align*}
  \end{enumerate}

This concludes our proof.
\end{proof}

\medskip

If $(\A,\Lip)$ is a {\qcms},  \cite[Theorem 4.1]{Rieffel99} demonstrates, using the fact that  $\Lip$ is lower semicontinuous,  that a $\ast$-automorphism $\alpha\in\Aut{\A}$ of $\A$ is $K$-bi-Lipschitz  for $(\A,\Lip)$ if, and only if, $\alpha^\ast : \varphi \rightarrow \varphi\circ\alpha$ is a $K$-bi-Lipschitz map on $(\StateSpace(\A),\Kantorovich{\Lip})$, i.e.,
\begin{equation*}
  \forall\varphi,\psi \in \StateSpace(\A) \quad \frac{1}{K}\Kantorovich{\Lip}(\varphi,\psi) \leq \Kantorovich{\Lip}(\varphi\circ\alpha,\psi\circ\alpha) \leq K \Kantorovich{\Lip}(\varphi,\psi) \text.
\end{equation*}
From this it is  easy to see that, in particular, the composition of full quantum isometries of $(\A,\Lip)$ is again a full quantum isometry of $(\A,\Lip)$.

The first group of interest in this work is the \emph{group of isometries of a {\qcms}}, which we now define.

\begin{definition}\label{Large-iso-def}
  The group $\LargeIso(\A,\Lip)$ of {\it isometries of a {\qcms}} $(\A,\Lip)$ is the group, under composition, of full quantum isometries from $(\A,\Lip)$ to itself.
\end{definition}
By \cite{Rieffel00}, for a {\qcms} $(\A,\Lip)$, since $\Lip$ is assumed lower semi-continuous over $\sa{\A}$, we note that
\begin{equation*}
  \LargeIso(\A,\Lip) = \left\{ \alpha \in \Aut{\A} : \forall \varphi,\psi \in \StateSpace(\A) \quad \Kantorovich{\Lip}(\varphi\circ\alpha,\psi\circ\alpha) = \Kantorovich{\Lip}(\varphi,\psi) \right\} \text.
\end{equation*}

Lemma (\ref{Kantorovich-Length-lemma}) proves that the restriction of $\KantorovichDist{\Lip}$ to $\LargeIso(\A,\Lip)$ is actually both left invariant (as it is induced by a length) and right invariant (as the length is invariant by conjugation).

\medskip

We now prove that the group of isometries of a {\qcms} is in fact compact for the Monge-Kantorovich metric of Definition (\ref{Kantorovich-morphism-def}). This was established in the special case of quantum metrics arising from spectral triples for twisted discrete finitely generated group C*-algebras in \cite{LoWu}. Our isometry groups are not defined in the same manner as Rieffel's isometry groups, since he works with order unit spaces instead of C$^\ast$-algebras, nonetheless, our proof is related to \cite[Proposition 6.6]{Rieffel00}. In fact, we obtain this result as a corollary of the following theorem about the compactness of sets of bi-Lipschitz automorphisms.

\begin{theorem}\label{compact-Iso-thm}
  If $(\A,\Lip)$ is a {\qcms}, and if $K \geq 1$, then the set
  \begin{equation*}
    \Aut{\A,\Lip,K} \coloneqq \left\{ \alpha \in \Aut{\A} : \text{$\alpha$ is $K$-bi-Lipschitz} \right\}
  \end{equation*}
  is compact for $\KantorovichDist{\Lip}$.
\end{theorem}

\begin{proof}
  By Theorem (\ref{KantorovichDist-thm}), the space $\Aut{\A,\Lip,K}$ is totally bounded for $\KantorovichDist{\Lip}$. It is then sufficient to prove that $\Aut{\A,\Lip,K}$ is complete for $\KantorovichDist{\Lip}$. Thus, let $(\alpha_n)_{n\in\N}$ be a Cauchy sequence in $\Aut{\A,\Lip,K}$ for $\KantorovichDist{\Lip}$. Set
  \begin{equation*}
    B \coloneqq \left\{ a \in \dom{\Lip} : \Lip(a) \leq 1,\norm{a}{\A} \leq \qdiam{\A}{\Lip} \right\}\text.
  \end{equation*}

  If $\mu \in \StateSpace(\A)$ and $a\in\dom{\Lip}$ with $\Lip(a)\leq 1$, then $\norm{a-\mu(a)}{\A} \leq \Lip(a)\qdiam{\A}{\Lip}$ by \cite{Rieffel98a}. Therefore, for any $\alpha,\beta\in\Aut{\A}$, we observe that
  \begin{equation*}
    \norm{\alpha(a)-\beta(a)}{\A} = \norm{\alpha(a-\mu(a))-\beta(a-\mu(a))}{\A} 
  \end{equation*}
 implies
  \begin{equation*}
    \KantorovichDist{\Lip}(\alpha,\beta) = \sup\left\{\norm{\alpha(a)-\beta(a)}{\A} : a\in B \right\} \text.
  \end{equation*}

  For any $a\in B$, the sequence $(\alpha_n(a))_{n\in\N}$ is Cauchy for $\norm{\cdot}{\A_\infty}$ since $(\alpha_n)_{n\in\N}$ is Cauchy for $\KantorovichDist{\Lip}$. Since $\A$ is complete, the sequence $(\alpha_n(a))_{n\in\N}$ converges.

  Moreover, if $a\in\dom{\Lip}$ with $\Lip(a) \leq 1$, then $(\alpha_n(a-\mu(a)))_{n\in\N}$ converges, since $a-\mu(a) \in B$; as $\alpha(\mu(a))=\mu(a)$, we conclude that $(\alpha_n(a))_{n\in\N}$ converges, as well.

  For $a\in\dom{\Lip}$, the sequence $\left(\alpha_n\left(\frac{a}{\Lip(a)+1}\right)\right)_{n\in\N}$ converges, since $\Lip\left(\frac{a}{\Lip(a)+1}\right)\leq 1$. Therefore, the sequence $(\alpha_n(a))_{n\in\N}$ converges as well, since $\alpha_n$ is linear for all $n\in\N$. Thus, the sequence $(\alpha_n)_{n\in\N}$ of $\ast$-automorphisms converges pointwise on $\dom{\Lip}$; let $\alpha$ be its pointwise limit.
  
  Define $A \coloneqq\{ a+ib : a,b \in \dom{\Lip} \}$. Since $\dom{\Lip}$ is dense in $\sa{\A}$, it can be  easily checked that $A$ is dense in $\A$. Since $\dom{\Lip}$ is a Jordan-Lie algebra, we also easily check that $A$ is a $\ast$-subalgebra of $\A$. We extend $\alpha$ to $A$ simply by setting $\alpha(a+ib) \coloneqq \alpha(a) + i\alpha(b)$ --- noting that $\alpha$ is indeed well-defined this way.

  Furthermore, if $a+ib \in A$, then $(\alpha_n(a))_{n\in\N}$ and $(\alpha_n(b))_{n\in\N}$ converge to, respectively, to $\alpha(a)$ and $\alpha(b)$; thus
  \begin{equation*}
    \alpha_n(a+ib) = \alpha_n(a) + i\alpha_n(b) \xrightarrow{n\rightarrow\infty} \alpha(a) + i\alpha(b) = \alpha(a+ib) \text.
  \end{equation*}
  Therefore $\alpha$ is the pointwise limit of $(\alpha_n)_{n\in\N}$ on $A$; as such, $\alpha$ is a $\ast$-morphism on $A$.
Moreover, since $\alpha_n$ is a $\ast$-automorphism for every $n\in\N$, it is an isometry, so we have, for all $a\in \dom{\Lip}$,
  \begin{equation*}
    \norm{\alpha(a)}{\A} = \lim_{n\rightarrow\infty} \norm{\alpha_n(a)}{\A} = \norm{a}{\A} \text.
  \end{equation*}
  Thus $\alpha$ is uniformly continuous over $A$, and so it admits a unique uniformly continuous extension of $\alpha$ to $\A$; we still denote this extension by $\alpha$. Of course, $\alpha$ is again a $\ast$-morphism of $\A$.

  Now, the sequence $(\alpha_n)_{n\in\N}$ is a Cauchy sequence for the supremum distance over $B$, namely $\Kantorovich{\Lip}$, and it converges pointwise to $\alpha$ on $B$, so a standard argument shows that $\alpha$ is the uniform limit of $(\alpha_n)_{n\in\N}$ over $B$, i.e.,
  \begin{equation*}
    \lim_{n\rightarrow\infty} \KantorovichDist{\Lip}(\alpha_n,\alpha) = 0 \text.
  \end{equation*}

  By Lemma (\ref{Kantorovich-Length-lemma}), since $\alpha_n\in\Aut{\A,\Lip,K}$ for all $n\in\N$, the sequence $(\alpha_n^{-1})_{n\in\N}$ is Cauchy for $\KantorovichDist{\Lip}$ as well. Following the same reasoning as above, let $\beta$ be the pointwise limit of $(\alpha_n^{-1})_{n\in\N}$, and fix  $\varepsilon > 0$. Since $(\alpha_n^{-1})_{n\in\N}$ converges uniformly to $\beta$ on $B$, we conclude that there exists $N\in\N$ such that, if $n\geq N$, then $\KantorovichDist{\Lip}(\alpha_n^{-1},\beta)<\frac{\varepsilon}{2 K}$. Similarly, there exists $N' \in \N$ such that if $n\geq N'$ then $\KantorovichDist{\Lip}(\alpha_n,\alpha) < \frac{\varepsilon}{2}$. Thus, if $a\in B$ and $n\geq\max\{N,N'\}$, we have:
  \begin{align*}
    \norm{a-\beta(\alpha(a))}{\A}
    &\leq \norm{a-\beta(\alpha_n(a))}{\A} + \norm{\beta(\alpha_n(a)-\alpha(a))}{\A} \\
    &\leq \norm{a-\alpha_n^{-1}(\alpha_n(a))}{\A} + \norm{\alpha_n^{-1}(\alpha_n(a))-\beta(\alpha_n(a))}{\A} + \norm{\beta(\alpha_n(a)-\alpha(a))}{\A} \\
    &\leq 0 + K\norm{\alpha_n^{-1}(\underbracket[1pt]{K^{-1}\alpha_n(a)}_{\Lip(K^{-1}\alpha_n(a))\leq 1})-\beta(K^{-1}\alpha_n(a))}{\A} + \underbracket[1pt]{\norm{\alpha_n(a)-\alpha(a)}{\A}}_{\beta\text{i is  $\ast$-morphism}} \\
    &\leq 0 + K\, \KantorovichDist{\Lip}(\alpha_n^{-1},\beta) + \KantorovichDist{\Lip}(\alpha_n,\alpha) \\
    &< 0 + \frac{\varepsilon}{2} + \frac{\varepsilon}{2} = \varepsilon \text.
  \end{align*}
  So $a = \beta(\alpha(a))$ as $\varepsilon > 0$ was arbitrary. By symmetry, $\beta(\alpha(a)) = a$ as well. So $\alpha$ is indeed a $\ast$-automorphism of $\A$.
Fix  $a\in \dom{\Lip}$. Since $\Lip$ is lower semicontinuous, $\Lip(\alpha(a)) \leq \liminf_{n\rightarrow\infty}\Lip(\alpha_n(a)) \leq K\, \Lip(a)$, and the same reasoning applies to $\alpha^{-1}$. So, for all $a\in\dom{\Lip}$:
  \begin{equation*}
    \frac{1}{K} \Lip(a) = \frac{1}{K} \Lip(\alpha^{-1}(\alpha(a))) \leq \Lip(\alpha(a)) \leq K \Lip(a) \text.
  \end{equation*}
  Now, since both $\alpha$ and $\alpha^{-1}$ map $\dom{\Lip}$ to itself, we have
  \begin{equation*}
    \dom{\Lip} = \alpha(\alpha^{-1}(\dom{\Lip})) \subseteq \alpha(\dom{\Lip}) \subseteq \dom{\Lip}
  \end{equation*}
  so $\alpha(\dom{\Lip}) = \dom{\Lip}$.
  
  Summarizing,  $\Aut{\A,\Lip,K}$ is complete and totally bounded, so it is compact for $\KantorovichDist{\Lip}$.
\end{proof}

\begin{corollary}\label{cor:ISO-compact}
  If $(\A,\Lip)$ is a {\qcms}, then $\LargeIso(\A,\Lip)$ is a compact group in $\KantorovichDist{\Lip}$.
\end{corollary}

\begin{proof}
  We apply Theorem (\ref{compact-Iso-thm}) to $\LargeIso(\A,\Lip) = \Aut{\A,\Lip,1}$, which implies that $\LargeIso(\A,\Lip)$ is a compact space. By Lemma (\ref{Kantorovich-Length-lemma}), since $\KantorovichLength{\Lip}$ is a length function on $\LargeIso(\A,\Lip)$, composition is continuous as well. Since $\LargeIso(\A,\Lip)$ is compact and has continuous multiplication, the inverse map is also continuous: in fact, Lemma (\ref{Kantorovich-Length-lemma}) proves that the inverse map is an isometry.
\end{proof}

We conclude this section with an observation, to illustrate the value of some of our results. In general, for the topology of pointwise convergence on the set of $\ast$-morphisms of any unital C*-algebra $\A$, the set $\Aut{\A}$ is not a closed subset. For instance, let $\A$ be the C*-algebra of all convergent sequences with values in $\C$. For each $n\in\N$, and for each $(x_k)_{k\in\N} \in \A$, we define
\begin{equation*}
  \alpha_n( (x_k)_{k\in\N} ) \coloneqq \left(x_n,x_0,x_1,x_2,\ldots,x_{n-1},x_{n+1},x_{n+2},\ldots\right) \text.
\end{equation*}
It is easy to check that $\alpha_n$ is a $\ast$-automorphism of $\A$ with inverse:
\begin{equation*}
  \alpha_n( (x_k)_{k\in\N} ) \coloneqq \left(x_1,x_2,x_3,x_4,\ldots,x_{n-1},x_{0},x_{n},\ldots\right) \text.
\end{equation*}
Since elements of $\A$ are convergent sequences, we then get that
\begin{equation*}
  \lim_{n\rightarrow\infty} \alpha_n( (x_k)_{k\in\N} ) = \beta( (x_k)_{k\in\N} ) \coloneqq \left(\lim_{k\rightarrow\infty} x_k, x_0, x_1, x_2, \ldots \right) \text.
\end{equation*}
So in particular, the sequence $\left(\frac{1}{n+1}\right)_{n\in\N}$ is not in the range of the limit $\ast$-morphism $\beta$.

On the other hand, $\Aut{\A}$ is a topological group. Indeed, if $(\alpha_n)_{n\in\N}$ in $\Aut{\A}$ converges pointwise to $\alpha$, and if $\alpha$ is a $\ast$-automorphism of $\A$ as well, then for all $a\in\A$, we have
\begin{equation*}
  \norm{\alpha_n^{-1}(a)-\alpha^{-1}(a)}{\A} = \norm{\alpha_n^{-1}(a-\alpha_n(\alpha^{-1}(a)))}{\A} = \norm{a-\alpha_n(\alpha^{-1}(a))}{\A} \xrightarrow{n\rightarrow\infty} 0 \text.
\end{equation*}
Therefore, the inverse map of $\Aut{\A}$ is indeed continuous.

Let us now add quantum metrics in this picture, and work with a {\qcms} $(\A,\Lip)$. We work here with a metric on $\Aut{\A}$, namely $\KantorovichDist{\Lip}$, and thus we can have a somewhat more precise understanding of the situation above. Indeed, the group $\Aut{\A}$ is not complete for this metric --- which metrizes the topology of pointwise convergence. However, we see that if we restrict our attention to sets of $K$-bi-Lipschitz automorphisms, then thanks to Theorem (\ref{compact-Iso-thm}),  the pointwise limit of a sequence of $K$-bi-Lipschitz $\ast$-automorphisms is indeed a $\ast$-automorphism (which is again $K$-bi-Lipschitz). Thus, adding the quantum metric data helps us identifying compact sets of automorphisms, and allows us to construct automorphisms as pointwise limits of other automorphisms.

\subsection{The Isometry Group of a Spectral Triple}\label{sec:isom-st}

For any spectral triple $(\A,\Hilbert,\Dirac)$, we define below the second group of interest in this work. Informally, this group is defined using morphisms in some category of spectral triples, rather than in the category of {\qcms s}. It is thus possible to define our second group without immediate reference to the metric property of the spectral triple. This group was introduced early in the study of noncommutative metric geometry by E. Park in \cite{Efton95}, where it is computed, for instance, for the noncommutative $2$-tori, and for   some spectral triples introduced by Connes on discrete groups in \cite{Connes89}. It is also shown in \cite[Theorem 1.2]{Efton95} that this group does agree with the group of isometries of a Riemannian manifold for a standard spectral triple.

\begin{definition}[{\cite{Efton95}}]\label{small-iso-def}
  The group $\smallIso(\A,\Hilbert,\Dirac)$ of isometries for the spectral triple $(\A,\Hilbert,\Dirac)$ is the following subgroup of the group $\mathrm{Aut}(\A)$ of $\ast$-automorphisms of $\A$:
  \begin{multline*}
    \smallIso(\A,\Hilbert,\Dirac) \coloneqq \\ \left\{ \alpha \in \mathrm{Aut}(\A) : \exists U \text{ unitary on $\Hilbert$ s.t. \ } \AdRep{U} = \alpha, U\dom{\Dirac}=\dom{\Dirac}, [\Dirac,U] = 0 \right\} \text.
  \end{multline*}
\end{definition}

Indeed, $\smallIso(\A,\Hilbert,\Dirac)$ is a subgroup of $\Aut{\A}$. For, fix $\alpha \in \smallIso(\A,\Hilbert,\Dirac)$. By Definition (\ref{small-iso-def}), there exists some unitary $U$ on $\Hilbert$ which commutes with $\Dirac$ such that $\alpha = \AdRep{U}$. Of course, $U^\ast=U^{-1}$ also commutes with $\Dirac^\ast=\Dirac$, and $\alpha^{-1} = \AdRep{U^\ast}$, so $\alpha^{-1} \in \smallIso(\A,\Hilbert,\Dirac)$. Moreover, if $\beta\in \smallIso(\A,\Hilbert,\Dirac)$, then $\beta=\AdRep{V}$ for some unitary $V$ on $\Hilbert$ which also commutes with $\Dirac$, and thus $\beta\circ\alpha^{-1} = \AdRep{VU^\ast}$; as $VU^\ast$ is a unitary commuting with $\Dirac$, we conclude that $\beta\circ\alpha^{-1} \in \smallIso(\A,\Hilbert,\Dirac)$. Since the identity is obviously in $\smallIso(\A,\Hilbert,\Dirac)$, the set $\smallIso(\A,\Hilbert,\Dirac)$ is indeed a subgroup of the group of $\ast$-automorphisms of $\A$.

It is also immediate that if $a \in \A$ with $a\,\dom{\Dirac}\subseteq\dom{\Dirac}$ and $[\Dirac,a]$ is bounded, then we have:
\begin{align}\label{spectral-triple-iso-comm-eq}
  \opnorm{[\Dirac,\alpha(a)]}{}{\Hilbert}
  &= \opnorm{[\Dirac,U a U^\ast]}{}{\Hilbert} \\
  &= \opnorm{U[\Dirac,a]U^\ast}{}{\Hilbert} = \opnorm{[\Dirac,a]}{}{\Hilbert} \text, \nonumber
\end{align}
where $\opnorm{T}{}{\Hilbert}$ is the usual operator norm of any linear operator $T$ on $\Hilbert$.

Of course, Expression (\ref{spectral-triple-iso-comm-eq}) implies that the elements of $\smallIso(\A,\Hilbert,\Dirac)$ induce isometries for the Connes metric of $(\A,\Hilbert,\Dirac)$. It is then natural to work with spectral triples which give rise to {\qcms s}, called \emph{metric spectral triples}, which are exactly the spectral triples for which the Connes metric is actually a distance which induces the weak* topology on the state space of $\A$  \cite[Proposition 1.10]{Latremoliere18g}.

\begin{definition}\label{metric-spectral-triple-def}
  A spectral triple $(\A,\Hilbert,\Dirac)$ is \emph{metric} when $(\A,\Lip_\Dirac)$ is a {\qcms}, where
  \begin{equation*}
    \dom{\Lip_\Dirac} \coloneqq \left\{ a\in\sa{\A} : a\,\dom{\Dirac}\subseteq\dom{\Dirac}, [\Dirac,a]\text{ is bounded } \right\}
  \end{equation*}
  and for all $a \in \dom{\Lip_\Dirac}$ we define 
  \begin{equation*}
    \Lip_\Dirac(a) \coloneqq \opnorm{[\Dirac,a]}{}{\Hilbert} \text.
  \end{equation*}
  In particular, the state space $\StateSpace(\A)$ is a metric space with respect to the Monge-Kantorovich metric $\Kantorovich{\Lip_\Dirac}$ (induced by the seminorm ${\Lip_\Dirac}$), defined by, for all $\varphi, \psi \in \StateSpace(\A)$:
   \begin{equation}\label{def:MK-norm-state-space}
  \Kantorovich{\Lip_\Dirac}(\varphi,\psi) \coloneqq \sup\left\{ \norm{\varphi(a)-\psi(a)}{\A} : a\in\dom{\Lip_\Dirac}, \Lip_\Dirac(a) \leq 1 \right\} \text.
  \end{equation}
\end{definition}

\begin{remark}
  When $(\A,\Hilbert,\Dirac)$ is a metric spectral triple, the group which we denote here by $\LargeIso(\A,\Lip_\Dirac)$ was denoted by $\LargeIso(\A,\Hilbert,\Dirac)$ in \cite{Conti21}.
   Our results about the group of isometries of a {\qcms} in the previous subsection, however, are not dependent on the spectral triple structure, except for the induced metric, hence our more general notation.
\end{remark}

We simplify our notation somewhat when working with metric spectral triples, as explained below.
\begin{notation}\label{not-metric-st}
  If $(\A,\Hilbert,\Dirac)$ is a metric spectral triple, we simplify the notation in Expressions    \eqref{eq:def-MK-for-autos}
  \eqref{eq:def-of-mkell} and \eqref{def:MK-norm-state-space} by writing   $\KantorovichDist{\Dirac}$,  $\KantorovichLength{\Dirac}$ and  $\Kantorovich{\Dirac}$ for $\KantorovichDist{\Lip_\Dirac}$, $\KantorovichLength{\Lip_\Dirac}$,  and $\Kantorovich{\Lip_\Dirac}$, respectively. We also write $\qdiam{\A}{\Dirac}$ for $\qdiam{\A}{\Lip_{\Dirac}}$.
\end{notation}

If the spectral triple $(\A,\Hilbert,\Dirac)$ is metric, then it follows from Expression (\ref{spectral-triple-iso-comm-eq}) that $\smallIso(\A,\Hilbert,\Dirac)$ is a subgroup of $\LargeIso(\A,\Lip_\Dirac)$. In general, however, $\smallIso(\A,\Hilbert,\Dirac)$ is a proper subgroup of $\LargeIso(\A,\Lip_\Dirac)$. The following example establishes this point, and it is an adaptation of \cite[Proposition 2.4]{Conti21}, though the latter example is not that of a metric spectral triple.

\begin{example}\label{ex:Iso-diff-ISO}
  Let $\A_2$ be the C*-algebra $M_2(\C)$ of $2\times 2$ matrices over $\C$ represented on $\C^4$ via the $\ast$-representation $\pi : a\in M_2(\C) \mapsto \begin{pmatrix} a & 0 \\ 0 & a\end{pmatrix}$. We also define the following two unitary self-adjoint involutions: 
  \[S \coloneqq \begin{pmatrix} 0 & 1 \\ 1 & 0 \end{pmatrix} \quad \text{ and }\quad J \coloneqq \begin{pmatrix} 1 & 0 \\ 0 & -1 \end{pmatrix}\] 
  and the Dirac operator
  \begin{equation*}
    \Dirac \coloneqq \begin{pmatrix} S & 0 \\ 0 & J \end{pmatrix} \text.
  \end{equation*}
  Of course, $(\A_2,\C^4,\Dirac)$ is a spectral triple. Moreover, for  all  $a \in \A_2$, we have:
  \begin{equation}\label{eq:com--repre-A-2}
    [\Dirac,\pi(a)] = \begin{pmatrix} [S,a] & 0 \\ 0 & [J,a] \end{pmatrix},
  \end{equation}
  which implies that  $[\Dirac,\pi(a)] = 0$ if, and only if, $a$ commutes with both $S$ and  $J$, which is equivalent to 
  
  \[
  \begin{split}
  a&=\begin{pmatrix} x & y \\ y & x \end{pmatrix}, \text{ for some } x,y \in \C, \text{ and }\\
  a &= \begin{pmatrix} x' & 0 \\ 0 & y'  \end{pmatrix} \text{ for some } x',y'\in\C.
  \end{split}
  \]
Therefore $[\Dirac,\pi(a)] = 0$ if, and only if $a=x 1$ for some $x\in\C$.
Thus, setting $\Lip_\Dirac(a) \coloneqq \opnorm{[\Dirac,\pi(a)]}{}{\C^4}$ for all $a\in\A$, the pair $(\A,\Lip_\Dirac)$ is a {\qcms}.

Now, a  straightforward calculation shows that $[S,J a J] = -J[S,a]J$. So if we define the unitary involution $U \coloneqq \pi(J)$,  then  by Equation \eqref{eq:com--repre-A-2} we have that, for all $a\in\A_2$:
  \begin{equation*}
    [\Dirac,\pi(J a J^\ast)] = [\Dirac,U\pi(a)U^\ast] = \begin{pmatrix} -J[S,a]J & 0 \\ 0 & -[J,a] \end{pmatrix}.
  \end{equation*}
  Therefore  $\Lip_\Dirac(J a J) = \Lip_\Dirac(a)$, which  shows that  $\AdRep{J}$ is a full quantum isometry, i.e., $\AdRep{J} \in \LargeIso(\A,\Lip_\Dirac)$.

We also note that $[\Dirac,U]\neq 0$ since $[S,J]\neq 0$. 
In addition,  we claim that there exists no unitary $V \in M_4(\C)$ which commutes with $\Dirac$ and implements $\AdRep{U}$ on $\pi(\A_2)$. 

Indeed, assume by contradiction that   $V$ is  a unitary in $M_4(\C)$, with  $\AdRep{V}$ acting on $\pi(\A_2)$ as $\AdRep{U}$. Thus, $\AdRep{V U^\ast}$ acts as the identity on $\pi(\A_2)$. Now, if $W = \begin{pmatrix} a_1 & a_2 \\ a_3 & a_4 \end{pmatrix}$, with $a_1,a_2,a_3,a_4\in \A_2$, is some unitary such that $\AdRep{W}$ acts trivially on $\pi(\A_2)$, then $W\begin{pmatrix} a & 0 \\ 0 & a \end{pmatrix} = \begin{pmatrix} a & 0 \\ 0 & a \end{pmatrix} W$, and thus $[a_j,a] = 0$ for $j\in\{1,2,3,4\}$, for all $a\in\A_2$. Therefore, $a_1,a_2,a_3,a_4 \in \C \, 1$, which implies:
\[V = \begin{pmatrix} a_1 J & a_2 J \\ a_3 J & a_4 J \end{pmatrix}.\]

  Now, if $[\Dirac,V] = 0$, then we also must have $S a_2 J = a_2 J^2 = a_2$. Since $SJ\notin \C$, we conclude that $a_2 = 0$. Similarly, $a_3 = 0$. Moreover, $S a_1 J = a_1 J S$ implies $a_1 = 0$ as well. So $V$ can not be a unitary.

  Therefore, no unitary exists which commute with $\Dirac$ and implements $\AdRep{U}$ on $\pi(\A_2)$, which implies that $\AdRep{U} \in \LargeIso(\A,\Lip)$ but $\AdRep{U} \notin \smallIso(\A,\Hilbert,\Dirac)$.
\end{example}

The isometry group of a metric spectral triple, while it may be a proper subgroup of the isometry group of the underlying {\qcms}, is always a closed subgroup for the Monge-Kantorovich distance of Definition (\ref{Kantorovich-morphism-def}).

\begin{theorem}\label{compact-iso-group-thm}
  The group $\smallIso(\A,\Hilbert,\Dirac)$ is a closed subgroup in $\LargeIso(\A,\Lip_\Dirac)$ --- in particular, it is compact --- for the metric $\KantorovichDist{\Dirac}$.
\end{theorem}

\begin{proof}
   As we already noted, by  Expression (\ref{spectral-triple-iso-comm-eq}) it is straightforward to see that $\smallIso(\A,\Hilbert,\Dirac)$ is a subgroup of $\Aut{\A}$ included in  $\LargeIso(\A,\Lip_\Dirac)$. It is thus sufficient to prove that $\smallIso(\A,\Hilbert,\Dirac)$ is closed in the group $\LargeIso(\A,\Lip_\Dirac)$, which is compact for $\KantorovichDist{\Dirac}$ by Theorem (\ref{compact-Iso-thm}).
  To do so, define,   for all $\xi \in \dom{\Dirac}$:  
 \begin{equation*}
 \CDN(\xi) \coloneqq \norm{\xi}{\Hilbert} + \norm{\Dirac\xi}{\Hilbert}
 \end{equation*}
Let $(\alpha_n)_{n\in\N}$ be a sequence in $\smallIso(\A,\Hilbert,\Dirac)$ converging for the Monge-Kantorovich metric  $\KantorovichDist{\Dirac}$ to $\alpha \in \LargeIso(\A,\Lip_\Dirac)$. Our goal is to prove that $\alpha \in \smallIso(\A,\Hilbert,\Dirac)$.

By definition, for each $n\in\N$, there exists a unitary $U_n\in \mathcal{B}(\Hilbert)$ such that $U_n\dom{\Dirac}=\dom{\Dirac}$, $[\Dirac,U_n] = 0$ and $\alpha_n = \AdRep{U_n}$. Define 
  \[B \coloneqq \left\{ \xi \in \dom{\Dirac} : \CDN(\xi)\leq  1\right\}.\] 
  As $\Dirac$ has compact resolvent, the set $B$ is compact in $\Hilbert$ by \cite[Theorem 2.7]{Latremoliere18g}. Since $U_n\dom{\Dirac}=\dom{\Dirac}$ for all $n\in\N$, we have, for all $\xi \in B$:
  \begin{align*}
    \CDN(U_n\xi)
    &= \norm{U_n\xi}{\Hilbert} + \norm{\Dirac U_n\xi}{\Hilbert} \\
    &= \underbracket[1pt]{\norm{\xi}{\Hilbert}}_{U_n \text{ is unitary}} + \underbracket[1pt]{\norm{U_n\Dirac \xi}{\Hilbert}}_{[\Dirac,U_n] = 0} \\
    &= \norm{\xi}{\Hilbert} + \norm{\Dirac\xi}{\Hilbert} = \CDN(\xi)\text.
  \end{align*}

  Thus $U_n B \subseteq B$. 

  Therefore the set $\{ U_n : n \in \N \}$ is an equicontinuous subset of the set $C(B,B)$ of $B$-valued continuous functions over $B$ in the standard operator norm $\opnorm{\cdot}{}{B}$. Since  $B$ is  compact,  by Arzel\`a-Ascoli's theorem, $\{ U_n : n \in \N \}$ is totally bounded for $\opnorm{\cdot}{}{B}$. Thus, there exists a Cauchy subsequence $(U_{f(n)})_{n\in\N}$ of $(U_n)_{n\in\N}$ with respect to  $\opnorm{\cdot}{}{B}$. Since $C(B,B)\subseteq C(B,\Hilbert)$ and $\Hilbert$ is complete, $C(B,\Hilbert)$ is complete for $\opnorm{\cdot}{}{B}$. Therefore, the Cauchy sequence $(U_{f(n)})_{n\in\N}$ converges (uniformly) in $C(B,\Hilbert)$  with respect to $\opnorm{\cdot}{}{B}$. Let $V$ denote its limit --- so far, $V$ is only defined on $B$; since $B$ is closed, we note that $V B\subseteq B$.

  By definition of $\opnorm{\cdot}{}{B}$, for all $\xi \in B$, we have in particular that $\lim_{n\rightarrow\infty} \norm{U_n\xi - V\xi}{\Hilbert} = 0$. Now, let $\xi \in \dom{\Dirac}\setminus\{0\}$. Since $\frac{1}{\CDN(\xi)}\xi \in B$ and $U_n$ is linear for all $n\in\N$, we conclude that $(U_{f(n)}\xi)_{n\in\N}$ converges in norm to $\CDN(\xi)V(\CDN(\xi)^{-1}\xi)$ --- we denote the latest expression by $V\xi$. Thus, since  $V$ is the pointwise limit of the sequence $(U_{f(n)})_{n\in\N}$ of linear maps, $V$ is also linear on $\dom{\Dirac}$. Since $U_n$ is unitary for all $n\in\N$, for all $\xi \in \dom{\Dirac}$, we have $\norm{U_n\xi}{\Hilbert} = \norm{\xi}{\Hilbert}$, and therefore, we conclude that $\norm{V\xi}{\Hilbert} = \norm{\xi}{\Hilbert}$. So $V$ can be extended to $\Hilbert$ as a continuous linear map of norm $1$.

  Because $\dom{\Dirac}$ is dense,  fixed  $\xi \in \Hilbert$ and $\varepsilon > 0$, there exists $\eta\in\dom{\Dirac}$ such that $\norm{\xi-\eta}{\Hilbert} < \frac{\varepsilon}{3}$, Moreover, since $ U_{f(n)} \to V$ uniformly on $B$,  there exists $N\in\N$ such that for all $n\geq N$, we have $\norm{U_{f(n)}\eta-V\eta}{\Hilbert} < \frac{\varepsilon}{3}$. Therefore: \[\norm{U_{f(n)}\xi-V\xi}{\Hilbert}\leq \norm{U_{f(n)}(\xi-\eta)}{\Hilbert} + \norm{U_{f(n)}\eta-V\eta}{\Hilbert} + \norm{V(\eta-\xi)}{\Hilbert} < \varepsilon.\]
  
 So $V$ is the SOT-pointwise limit of the sequence $(U_{f(n)})_{n\in\N}$. In particular, $V$ is an isometry, since $\norm{V\xi}{\Hilbert} = \lim_{n\rightarrow\infty} \norm{U_{f(n)}\xi}{\Hilbert}$.

  By Lemma (\ref{Kantorovich-Length-lemma}), we note that $(\alpha_n^{-1})_{n\in\N}$ converges to $\alpha^{-1}$ for $\KantorovichDist{\Lip}$. Of course, $\alpha_n^{-1} = \AdRep{U_n^\ast}$ for all $n\in\N$. Moreover, $U_n^\ast\dom{\Dirac}=\dom{\Dirac}$ and $[\Dirac,U_n^\ast]=-[\Dirac,U_n]=0$ since $\Dirac$ is self-adjoint. By applying the above methods, we conclude that there exists some convergent subsequence $(U_{f(g(n))}^\ast)_{n\in\N}$ of $(U_{f(n)}^\ast)_{n\in\N}$ convergent for the SOT topology, with respect to  $\opnorm{\cdot}{}{B}$. From now on, we write $h = f\circ g$. Let $W$ be the SOT limit of $(U_{h(n)}^\ast)_{n\in\N}$; $W$ is again an isometry on $\Hilbert$. As above, we also note that $W B \subseteq B$.

  By putting together all of the above results,  we see that for any fixed $\varepsilon>0$,  there exists $N \in \N$ such that if $n\geq N$, then $\opnorm{U_{h(n)}-V}{}{B} < \frac{\varepsilon}{2}$, and there exists $N' \in \N$ such that, if $n\geq N'$, then $\opnorm{U_{h(n)}^\ast - W}{}{B}<\frac{\varepsilon}{2}$, so for all $n\geq \max\{N,N'\}$  and  for all $\xi \in B$:
  
  \begin{align*}
    \norm{\xi-VW \xi}{\Hilbert}
    &\leq \norm{\xi- U_{h(n)}W\xi}{\Hilbert} + \underbracket[1pt]{\norm{U_{h(n)} W\xi - V W\xi}{\Hilbert}}_{\text{note: }W\xi \in B} \\
    &\leq \norm{\xi-U_{h(n)} U_{h(n)}^\ast \xi}{\Hilbert} + \norm{U_{h(n)} (U_{h(n)}^\ast - W)\xi}{\Hilbert} + \opnorm{U_{h(n)}-V}{}{B} \\
    &\leq 0 + \opnorm{U_{h(n)}^\ast - W}{}{B} + \opnorm{U_{h(n)} - V}{}{B} \\
    &< \frac{\varepsilon}{2} + \frac{\varepsilon}{2} = \varepsilon \text.
  \end{align*}
  Thus $VW\xi = \xi$ for all $\xi \in B$ since $\varepsilon > 0$ was arbitrary. Since $B$ is total in $\Hilbert$ and $VW$ is continuous and linear, we conclude that $VW\xi=\xi$ for all $\xi \in \Hilbert$. Hence, $V$ is surjective; since $V$ is an isometry, we conclude that $V$ is a unitary over $\Hilbert$.

  Moreover, fixed  $\xi \in \dom{\Dirac}$, the sequence $(U_{f(n)}\xi)_{n\in\N}$ converges to $V\xi$. On the other hand, for all $n\in\N$:
  \begin{equation*}
    \Dirac U_{f(n)}\xi = U_{f(n)} \Dirac\xi
  \end{equation*}
  converges to $V\Dirac\xi$. Since $\Dirac$ is closed, we conclude that $V\xi \in \dom{\Dirac}$ and $\Dirac V\xi = V\Dirac \xi$.  So $V\dom{\Dirac}\subseteq\dom{\Dirac}$ and $[\Dirac,V] = 0$.

  Our next goal is to prove that $\alpha = \AdRep{V}$. Fix $a \in \dom{\Lip}$ with $\Lip(a)\leq \frac{1}{2}$ and $\norm{a}{\A} \leq \frac{1}{2}$. If $\xi \in B$, then $\CDN(a\xi)\leq (\norm{a}{\A} + \Lip(a))\CDN(\xi) \leq 1$. So $a B \subseteq B$; moreover 
 $V^\ast\xi$ and $U_{f(n)}^\ast\xi$ are always also in $B$. 
 Fix $\varepsilon > 0$, and let $N\in\N$ such that, if $n\geq N$, then $\opnorm{U_{h(n)}-V}{}{B} < \frac{\varepsilon}{2}$ and $\opnorm{U_{h(n)}^\ast-V^\ast}{}{B} < \frac{\varepsilon}{2(\norm{a}{\A}+1)}$. Thus, for all $\xi \in B$ and $n\geq N$, we conclude:
  \begin{align*}
    \norm{U_{h(n)} a U_{h(n)}^\ast\xi - V a V^\ast\xi}{\Hilbert}
    &= \norm{(U_{h(n)} - V) aU_{h(n)}^\ast\xi}{\Hilbert}  + \norm{V a U_{h(n)}^\ast\xi - V a V^\ast\xi}{\Hilbert} \\
    &\leq \opnorm{U_{h(n)} - V}{}{B}\norm{a}{\A} + \norm{a}{\A} \opnorm{U_{h(n)}^\ast-V^\ast}{}{B} \\
    &< \frac{\varepsilon}{2} + \frac{\varepsilon}{2} = \varepsilon \text.
  \end{align*}
  Thus $\alpha$ and $\AdRep{V}$ agree on the total set $\left\{a\in\dom{\Lip}:\Lip(a)\leq \frac{1}{2},\norm{a}{\A}\leq\frac{1}{2}\right\}$ of $\A$. By continuity, $\alpha$ and $\AdRep{V}$ agree on $\A$.
  
  Thus we have shown that $\alpha \in \smallIso(\A,\Hilbert,\Dirac)$, as required. Thus $\smallIso(\A,\Hilbert,\Dirac)$ is closed in $\LargeIso(\A,\Lip)$ for $\KantorovichDist{\Lip}$, and therefore it is compact for this same metric since $\LargeIso(\A,\Lip)$ is closed by Theorem (\ref{compact-Iso-thm}).
\end{proof}

We conclude our section with two observations. First, just as we observed for $\ast$-automorphisms and pointwise convergence in the previous subsection, the SOT limit of unitaries is in general an isometry but not a unitary (the closure of the group of unitaries in the SOT over a Hilbert space is, in fact, the space of all linear isometries of that space). An example is given by essentially the same construction as we did with automorphisms: on $\ell^2(\N)$, for each $n\in\N$, we define $U_n((x_k)_{k\in\N})\coloneqq (x_n,x_0,x_1,x_2,\ldots,x_{n-1},x_{n+1},\ldots)$; the operator $U_n$ is a unitary for each $n\in\N$ and $(U_n)_{n\in\N}$ converges to the unilateral shift on $\ell^2(\N)$. So, once again, the addition of  quantum metrics allows us to work on compact collections of unitaries for the SOT, and guarantees that the SOT-limits of some sequences of unitaries are still unitaries.

Second, we note that Park proved in \cite[Theorem 4.2]{Efton95} that there exists a spectral triple $(C^\ast(S),\ell^2(\N),\Dirac)$, defined on the C*-algebra $C^\ast(S)$ generated by the unilateral shift $S$ on $\ell^2(\N)$, whose group $\smallIso(C^\ast(S),\ell^2(\N),\Dirac)$ is not compact in the topology of pointwise convergence. The Dirac operator $\Dirac$ is defined to be the closure of the operator defined by $\Dirac((x_n)_{n\in\N}) \coloneqq (n x_n)_{n\in\N}$ for all finitely supported $(x_n)_{n\in\N}\in\ell^2(\N)$.   As it is very easily checked, $\Dirac$ commutes with all the operators of the form $S^n (S^{\ast})^n$ for $n\in\N$, which span an infinite dimensional C*-subalgebra of $C^\ast(S)$. Therefore, $(C^\ast(S),\ell^2(S),\Dirac)$ is far from a metric spectral triple. Indeed, all other examples in \cite{Efton95} are constructed around metric spectral triples (before this concept was discovered), and thus give compact isometry groups. We see that this early observation on a few interesting examples is in fact a general phenomenon for \emph{metric} spectral triples.

\section{The Propinquity and Isometry Groups}\label{sec:prop-isom-gps}

We now turn to the matter of approximating isometry groups of {\qcms s} and metric spectral triples. In fact, we wish to approximate not only the groups themselves, but their actions on the underlying C*-algebras of {\qcms s}. To this end, the tool we choose is the extension of the Gromov-Hausdorff propinquity to certain proper metric group actions (and more generally, actions of proper monoids, though this will not be our focus here); this extension is called the covariant propinquity, introduced in \cite{Latremoliere18b} by the last author. We first see that the general theory for this metric has some interesting implications about isometry groups. We then turn to the specific situation where we work with inductive limits of {\qcms s} and of metric spectral triples, where a clearer picture emerges.

\medskip

In the rest of this paper, we will work with families of {\qcms s}, under the following assumption.
\begin{convention}
	In the rest of this paper, we fix the constants $\Omega \geq 1$, $\Omega'\geq 0$ in Definition (\ref{qcms-def}): all {\qcms s} considered here will satisfy this fixed $(\Omega,\Omega')$-quasi-Leibniz property.
\end{convention}

\subsection{The Covariant Propinquity}\label{sec:cov-Prop}

Our primary interest in {\qcms s} lies in our ability to construct a metric on the space of all {\qcms s}, called the \emph{(dual Gromov-Hausdorff) propinquity}, in a manner analogous to the Gromov-Hausdorff distance on the class of compact metric spaces.

\begin{definition}[{\cite{Latremoliere13b,Latremoliere14}}]\label{def:prop}
  The (dual) \emph{propinquity} $\dpropinquity{}$ between two {\qcms s} $(\A,\Lip_\A)$ and $(\B,\Lip_\B)$ is the nonnegative number:
  \begin{align*}
    \dpropinquity{}((\A,\Lip_\A),(\B,\Lip_\B)) = \inf\Big\{  &\,\max_{j\in\{1,2\}}\Haus{\Lip_\D}\left(\StateSpace(\D),\left\{\varphi\circ\pi_j:\varphi\in\StateSpace(\A_j)\right\}\right) : \\ &\, (\D,\Lip_\D) \text{ is a {\qcms}, and } \\ &\, \pi_j:(\D,\Lip_\D)\rightarrow(\A_j,\Lip_j) \text{ are quantum isometries for $j\in\{1,2\}$}  \Big\} \text,
  \end{align*}
  where $\Haus{\Lip_\D}$ is the Hausdorff distance induced by the {\MongeKant} $\Kantorovich{\Lip_\D}$ for any {\qcms} $(\D,\Lip_\D)$.
\end{definition}

\begin{theorem}[{\cite{Latremoliere13b}}]\label{gh-complete}
  The propinquity is a complete metric, up to full quantum isometry, on the class of {\qcms s}, whose restriction to classical compact metric spaces induces the topology of the Gromov-Hausdorff distance.
\end{theorem}

The \emph{covariant propinquity} $\covpropinquity{}$ extends the construction of the propinquity to actions of proper metric monoids on {\qcms s} via Lipschitz linear maps \cite{Latremoliere18b}; we focus here on actions by Lipschitz $\ast$-morphisms. We recall that a proper metric group $(G,\delta)$ is a group $G$ endowed with a left invariant metric $\delta$ such that $G$ is a topological group for the topology of $\delta$, and the closed balls of $\delta$ are compact.

\begin{notation}
	For any $r\geq 0$, and any metric group $G$, we denote the closed ball of center $r$ and centered at the unit $e_G$ of $G$ as $G[r]$.
\end{notation}

\begin{definition}\label{lip-ds}
  A \emph{Lipschitz dynamical system} $(\A,\Lip_\A,G,\delta,\alpha)$ is a $5$-tuple where $(\A,\Lip_\A)$ is a {\qcms}, $(G,\delta)$ is a proper monoid, and $\alpha$ is an action of $G$ on $(\A,\Lip_\A)$ by Lipschitz morphisms, such that
  \begin{enumerate}
  \item $\forall a\in\A \quad \forall g \in G \quad \lim_{h\rightarrow g} \norm{\alpha^h(a)-\alpha^g(a)}{\A} = 0$, i.e., $\alpha$ is strongly continuous,
  \item $\forall g \in G \quad \exists D > 0 \quad \exists x>0 \quad \forall h \in G \quad \delta(h,g) < x \implies \Lip_\A\circ\alpha^h \leq D \Lip_\A$, i.e., $\alpha$ has locally bounded dilation factor.
  \end{enumerate}
\end{definition}

\begin{notation}
  Let $(\A,\Lip_\A,G,\delta,\alpha)$ be a Lipschitz dynamical system. We will simply write $(\A,\Lip,G,\alpha)$ when the metric on $G$ is clear from the context, and even $(\A,\Lip,G)$ when the metric on $G$ is clear, and when $G$ is given as a group of automorphisms of $\A$, in which case the action is well understood to be the natural action of $G$ on $\A$.
\end{notation}

For our purpose, we simply recall from \cite{Latremoliere18b} that, for any fixed $n \in \N$  and for any $\varepsilon > 0$, 
\begin{equation*}
  \covpropinquity{}((\A_n,\Lip_n,G_n,\delta_n,\alpha_n),(\A_\infty,\Lip_\infty,G_\infty,\delta_\infty,\alpha_\infty)) < \varepsilon
\end{equation*}
whenever
\begin{enumerate}
\item there exists a {\qcms} $(\D,\Lip_\D)$ and there are two quantum isometries $\psi_n: (\D,\Lip_\D)\rightarrow(\A_n,\Lip_n)$ and $\pi_n : (\D,\Lip_\D)\rightarrow (\A_\infty,\Lip_\infty)$ such that
  \begin{equation*}
    \max\left\{\Haus{\Kantorovich{\Lip_\D}}(\StateSpace(\D),\psi_n^\ast\StateSpace(\A_n)), \Haus{\Kantorovich{\Lip_\D}}(\StateSpace(\D),\pi_n^\ast\StateSpace(\A_\infty)) \right\} < \varepsilon\text,
  \end{equation*}
\item there exist $r_n \geq \frac{1}{\varepsilon}$, and unital maps $f_n : G_\infty \rightarrow G_n$ and $\xi_n: G_n \rightarrow G_\infty$ with
  \begin{equation}\label{rn-epsilon-iso-eq-1}
    \sup_{g,g' \in G_\infty[r_n]}\sup_{h \in G_n[r_n]}\left|\delta_n(f_n(g) f_n(g'), h)-\delta_\infty(g g', \xi_n(h))\right| < \varepsilon \text,
  \end{equation}
  and
  \begin{equation}\label{rn-epsilon-iso-eq-2}
    \sup_{g,g'\in G_n[r_n]}\sup_{h \in G_\infty[r_n]}\left|\delta_\infty(\xi_n(g) \xi_n(g'), h)-\delta_n(g g', f_n(h))\right| < \varepsilon \text,
  \end{equation}
  i.e., $(f_n,\xi_n)$ is an $r_n$-local $\varepsilon$-almost isometric isomorphism in the sense of \cite[Definition 2.5]{Latremoliere18b},
\item the following holds:
  \begin{equation*}
    \sup_{\varphi\in\StateSpace(\A_\infty)} \; \inf_{\mu\in\StateSpace(\A_n)} \; \sup_{g \in G_\infty[r_n]}\Kantorovich{\Lip_\D}(\varphi\circ\alpha_\infty^g\circ\pi_n,\mu\circ\alpha_n^{f_n(g)}\circ\psi_n) < \varepsilon
  \end{equation*}
  and
  \begin{equation*}
    \sup_{\varphi\in\StateSpace(\A_n)} \; \inf_{\mu\in\StateSpace(\A_\infty)} \; \sup_{g \in G_n[r_n]}\Kantorovich{\Lip_\D}(\varphi\circ\alpha_n^g\circ\psi_n,\mu\circ\alpha_\infty^{\xi_n(g)}\circ\pi_n) < \varepsilon \text.
  \end{equation*}
\end{enumerate}

\begin{notation}\label{not:ep-iso-iso}
	We will call a $\frac{1}{\varepsilon}$-local $\varepsilon$-almost isometric isomorphism simply an $\varepsilon$-iso-iso, as done in \cite[Definition 2.17]{Latremoliere18g}.
\end{notation}
	
The covariant propinquity between two Lipschitz dynamical systems $(\A,\Lip_\A,G,\delta_G,\alpha)$ and $(\B,\Lip_\B,H,\delta_H,\beta)$ is zero if, and only if, there exists a full quantum isometry $\pi : (\A,\Lip_\A) \rightarrow (\B,\Lip_\B)$ and an isometric group isomorphism $\gamma : G\rightarrow H$ such that, for all $g \in G$, we have $\pi\circ\alpha^g = \beta^{\gamma(g)}\circ\pi$. Among other consequences of this (see \cite{Latremoliere18b} for more details), we note if $(\A_n,\Lip_n,G_n,\delta_n,\alpha_n)_{n\in\N}$ converges to $(\A,\Lip,G,\delta,\alpha)$ for the covariant propinquity, then both the isometry and the isomorphism class of the group $G$ is uniquely defined by this convergence.

\medskip

It is interesting to study the question of completeness for the covariant propinquity; many subclasses of Lipschitz dynamical systems are shown to be complete for $\covpropinquity{}$ in \cite{Latremoliere18c}.  We recall here the following corollary from \cite{Latremoliere18c}, as a useful example of a completeness result.

\begin{theorem}[{\cite[Corollary 4.6]{Latremoliere18b}}]\label{completeness-thm}
  For each $n\in\N$, let $(\A_n,\Lip_n,G_n,\delta_n,\alpha_n)$ be a Lipschitz dynamical system such that:
  \begin{enumerate}
  \item $(\A_n,\Lip_n,G_n,\delta_n,\alpha_n)_{n\in\N}$ is a Cauchy sequence for $\covpropinquity{}$,
  \item for all $\varepsilon > 0$, there exist $\omega>0$ and $N\in\N$ such that, for all $n\geq N$, if $g,h\in G_n$ with $\delta_n(g,h)<\omega$, then $\delta_n(g^{-1},h^{-1}) < \varepsilon$,
  \item for all $\varepsilon > 0$, there exist $\omega>0$ with $N\in\N$ such that, for all $n\geq N$, if $g,h\in G_n$ and $\delta_n(g,h)<\omega$, then $\KantorovichDist{\Lip_n}(\alpha_n^g,\alpha_n^h) < \varepsilon$,
  \item there exists a locally bounded function $D : [0,\infty)\rightarrow[1,\infty)$ such that, for all $n\in\N$ and for all $g\in G_n$, we have $\Lip_n\circ\alpha_n^g\leq D(\delta_n(e_n,g))\Lip_n$.
  \end{enumerate}

  Then there exists a Lipschitz dynamical system $(\A,\Lip,G,\delta,\alpha)$ such that
  \begin{equation*}
    \lim_{n\rightarrow\infty}\covpropinquity{}((\A_n,\Lip_n,G_n,\delta_n,\alpha_n),(\A,\Lip,G,\delta,\alpha)) = 0 \text.
  \end{equation*}
\end{theorem}

\medskip

We now return to the matter of the isometry groups, where the following general result holds:
\begin{corollary}\label{cor:iso-limit}
  If $(\A_n,\Lip_n,G_n,\delta_n,\alpha_n)_{n\in\N}$ is a Cauchy sequence of Lipschitz dynamical systems for $\covpropinquity{}$, where, for each $n\in\N$, we assume that $G_n$ is a closed subgroup of the group $\LargeIso(\A_n,\Lip_n)$, and that $\delta_n = \KantorovichDist{\Lip_n}{}{}$, then there exists a proper metric group $G$ and an action $\alpha$ of $G$ by full quantum isometries on $(\A,\Lip)$ such that
  \begin{equation*}
    \lim_{n\rightarrow\infty} \covpropinquity{}((\A_n,\Lip_n,G_n,\KantorovichDist{\Lip_n},\alpha_n), (\A,\Lip,G,\KantorovichDist{\Lip},\alpha)) = 0 \text.
  \end{equation*}
\end{corollary}

\begin{proof}
  Note that $\delta_n$ is indeed left invariant by Lemma (\ref{Kantorovich-Length-lemma}). Since, for all $n\in\N$, the group $\LargeIso(\A_n,\Lip_n)$ is compact, the closed subgroup $G_n$ is also compact, hence proper. With our choice of metric, the third condition of Theorem (\ref{completeness-thm}) is trivially met. The second condition of Theorem (\ref{completeness-thm}) is also met by Lemma (\ref{Kantorovich-Length-lemma}).
  
  We note that $\Lip_n\circ\alpha_n^g = \Lip_n$ for all $g \in G_n$ and for all $n\in\N$, so the fourth condition of Theorem (\ref{completeness-thm}) is satisfied.
  
  Therefore, for actions of closed groups of quantum isometry, we conclude that Theorem  (\ref{completeness-thm}) applies. Thus, there exists a proper metric group $G$ and an action $\alpha$ of $G$ on $(\A,\Lip)$ by Lipschitz $\ast$-endomorphisms such that
  \begin{equation*}
    \lim_{n\rightarrow\infty} \covpropinquity{}((\A_n,\Lip_n,G_n,\alpha_n,\KantorovichDist{\Lip_n}), (\A,\Lip,G,\alpha,\KantorovichDist{\Lip})) = 0 \text.
  \end{equation*}
  
 	By \cite[Theorem 4.2]{Latremoliere18c}, we also conclude that $\{\alpha^g : g \in G\}$ is indeed a subgroup of full quantum isometries of $(\A,\Lip)$.
 	
\end{proof}

\subsection{Covariant Bridge Builders}

We now want to obtain some examples of convergence of interesting closed subgroups of quantum isometries in the setting of inductive limit of C*-algebras. In \cite{FaLaPa}, the two last authors together with J. Packer proved a necessary and sufficient condition for the convergence, in the propinquity, of inductive limits of {\qcms s} in terms of certain $\ast$-automorphisms called \emph{bridge builders}, defined as follows.

\begin{definition}\label{bridge-builder-def}
  Let $\A_\infty \coloneqq \closure{\bigcup_{n\in\N}\A_n}$, where $(\A_n)_{n\in\N}$ is an increasing sequence of C*-subalgebras of the unital C*-algebra $\A_\infty$, and the unit of $\A_\infty$ lies in $\A_0$.
  
	A \emph{bridge builder} $\pi$ for $((\A_n,\Lip_n),(\A_\infty,\Lip_\infty))$ is a $\ast$-automorphism of $\A_\infty$ with the following property: for all $\varepsilon > 0$, there exists $N\in\N$ such that, if $n\geq N$, then
	\begin{equation}\label{bridge-builder-eq-1}
	\forall a \in \dom{\Lip_\infty} \quad \exists b \in \dom{\Lip_n}: \quad \Lip_n(b)\leq\Lip_\infty(a) \text{ and }\norm{\pi(a)-b}{\A_\infty} < \varepsilon \Lip_\infty(a)
	\end{equation}
	and
	\begin{equation}\label{bridge-builder-eq-2}
	\forall b \in \dom{\Lip_n} \quad \exists a \in \dom{\Lip_\infty} :\quad \Lip_\infty(a)\leq\Lip_n(b) \text{ and }\norm{\pi(a)-b}{\A_\infty} < \varepsilon \Lip_n(b)\text.		
	\end{equation}	
\end{definition}

With the notation of Definition (\ref{bridge-builder-def}), it was shown in \cite[Theorem 2.22]{FaLaPa} that if we assume the existence of some $M > 0$ such that for all $n\in\N$, we have $\frac{1}{M}\Lip_n \leq \Lip_\infty \leq M \Lip_n$ over $\dom{\Lip_n}$, then the existence of a bridge builder is equivalent to the convergence, in the sense of the propinquity, of $(\A_n,\Lip_n)_{n\in\N}$ to $(\A_\infty,\Lip_\infty)$. Moreover, if one finds a bridge builder which is a full quantum isometry, and if the quantum metrics are induced by metric spectral triples, then \cite[Theorem 3.17]{FaLaPa} proves that the spectral triples themselves converge, in the stronger sense of the \emph{spectral propinquity} --- we refer to \cite{Latremoliere18g,Latremoliere22} for the construction and properties of the spectral propinquity.

\medskip

We now wish to define a covariant version of bridge builders, which, in some natural cases, will play a similar role as bridge builders, but for the covariant propinquity. To this end, we will work within the following framework.

\begin{hypothesis}\label{cov-bridge-hyp}
  For each $n\in\Nbar \coloneqq \N\cup\{\infty\}$, let $(\A_n,\Lip_n)$ be a {\qcms}, such that $\A_\infty = \closure{\bigcup_{n\in\N} \A_n}$, where $(\A_n)_{n\in\N}$ is an increasing (for $\subseteq$) sequence of C*-subalgebras of $\A_\infty$, with the unit of $\A_\infty$ in $\A_0$.

  For each $n\in\Nbar$, let $\alpha_n$ be an action of a proper metric group $(G_n,\delta_n)$ on $(\A_n,\Lip_n)$ by full quantum isometries, and for all $n\in\N$, let $(f_n,\xi_n)$ be a $\varepsilon_n$-iso-iso from $G_\infty$ to $G_n$ in the sense of Notation \eqref{not:ep-iso-iso}, with $\lim_{n\rightarrow\infty}\varepsilon_n = 0$.  We set $r_n \coloneqq \frac{1}{\epsilon_n}$.

   Moreover, we assume that, for all $\varepsilon > 0$, there exist $\omega>0$ and $N\in\N$ such that, for all $n\geq N$, if $g,h\in G_n\left[\frac{1}{\varepsilon}\right]$ and $\delta_n(g,h)<\omega$, then $\KantorovichDist{\Lip_n}(\alpha_n^g,\alpha_n^h) < \varepsilon$.
\end{hypothesis}

The core concept for this section is given in the following definition.

\begin{definition}\label{cov-bridge-def}
  A \emph{covariant bridge builder} $\pi$ for the data detailed in Hypothesis (\ref{cov-bridge-hyp}) is a $\ast$-automorphism of $\A_\infty$ with the following property: for all $\varepsilon > 0$, there exists $N\in\N$ such that, if $n\geq N$, then: 
  \begin{enumerate}
  \item for every $a\in\dom{\Lip_\infty}$, there exists $b \in \dom{\Lip_n}$ with $\Lip_n(b)\leq\Lip_\infty(a)$ and
    \begin{equation*}
      \sup_{g \in G_\infty[r_n]} \norm{\pi(\alpha_\infty^g(a))-\alpha_n^{f_n(g)}(b)}{\A_\infty}  < \varepsilon \Lip_\infty(a) \text,
    \end{equation*}
  \item for every $b\in\dom{\Lip_n}$, there exists $a \in \dom{\Lip_\infty}$ with $\Lip_\infty(a)\leq\Lip_n(b)$ and 
    \begin{equation*}
      \sup_{g \in G_n[r_n]} \norm{\pi(\alpha_\infty^{\xi_n(g)}(a))-\alpha_n^{g}(b)}{\A_\infty}  < \varepsilon \Lip_n(b) \text.
    \end{equation*}
  \end{enumerate}	
\end{definition}

The apparent asymmetry in Definition (\ref{cov-bridge-lemma}) can partially be addressed with the following lemma, which will prove helpful in computations later.
\begin{lemma}\label{cov-bridge-lemma}
  Assume Hypothesis (\ref{cov-bridge-hyp}). Let $\pi$ be an automorphism of $\A_\infty$. Then if for  all $\varepsilon > 0$, there exists $N \in \N$ such that for all $n\geq N$,  $a\in\dom{\Lip_\infty}$ and for all  $b \in \dom{\Lip_n}$ with $\max\{\Lip_\infty(a),\Lip_n(b)\}\leq 1$ such that
  \begin{equation*}
    \sup_{g \in G_n[r_n]} \norm{\pi(\alpha_\infty^{\xi_n(g)}(a)) - \alpha_n^g(b)}{\A_\infty} < \frac{\varepsilon}{2}\text,
  \end{equation*}
  then
  \begin{equation*}
    \sup_{g\in G_\infty[r_n]}\norm{\pi(\alpha_\infty^g(a)) - \alpha_n^{f_n(g)}(b)}{\A_\infty} < \varepsilon \text.
  \end{equation*}
\end{lemma}

\begin{proof}
  Let $\varepsilon > 0$. By assumption on $\alpha_\infty$, there exists $\omega >0$ such that, for all $g,g' \in G_\infty$, if $\delta_\infty(g,g') < \omega$, then  $\KantorovichDist{\Lip_\infty}(\alpha_\infty^g,\alpha_\infty^{g'}) < \frac{\varepsilon}{2}$. On the other hand, there exists $N\in\N$ such that, if $n\geq N$, then $(f_n,\xi_n)$ is a $\frac{\varepsilon}{2}$-iso-iso. Therefore, in particular, by \cite[Lemma 2.10]{Latremoliere18b}, we have $\delta_\infty(g,\xi_n(f_n(g))) < \omega$ for all $g \in G_\infty[r_n]$. Therefore:
  \begin{align*}
    \norm{\pi(\alpha_\infty^g(a)) - \alpha_n^{f_n(g)}(b)}{\A_\infty}
    &\leq \norm{\pi(\alpha_\infty^g(a) - \alpha_\infty^{\xi_n(f_n(g))}(b))}{\A_\infty} \\
    &\quad + \norm{\pi(\alpha_\infty^{\xi_n(f_n(g))}(a)) - \alpha_n^{f_n(g)}(b)}{\A_\infty} \\
    &\leq \frac{\varepsilon}{2} + \frac{\varepsilon}{2} = \varepsilon \text.
  \end{align*}
  This completes our proof.
\end{proof}

Now, the existence of covariant bridge builders implies convergence for the covariant propinquity.

\begin{theorem}\label{covariant-bridge-builder-thm}
  For each $n\in\Nbar \coloneqq \N\cup\{\infty\}$, let $(\A_n,\Lip_n)$ be a {\qcms}, such that $\A_\infty = \closure{\bigcup_{n\in\N} \A_n}$, where $(\A_n)_{n\in\N}$ is an increasing (for $\subseteq$) sequence of C*-subalgebras of $\A_\infty$, with the unit of $\A_\infty$ in $\A_0$.
Assume that 	
 $(\A_n,\Lip_n)_{n\in\Nbar}$ are  {\qcms s},  and for each $n\in\Nbar$, let $\alpha_n$ be an action of a group $G_n$ on $(\A_n,\Lip_n)$ by full quantum isometries, satisfying Hypothesis (\ref{cov-bridge-hyp}).
  
  If there exists a covariant bridge builder for this data, then
  \begin{equation*}
    \lim_{n\rightarrow\infty} \covpropinquity{}((\A_n,\Lip_n,\alpha_n,G_n),(\A_\infty,\Lip_\infty,\alpha_\infty,G_\infty)) = 0 \text.
  \end{equation*}
\end{theorem}

\begin{proof}
  Let $\pi$ be a covariant bridge builder and let $\varepsilon > 0$. By assumption, there exists $\delta>0$ and $N_0\in\N$ such that, for all $n\in\Nbar$ with $n\geq N_0$, and for all $g,g' \in G_\infty\left[\frac{1}{\varepsilon}\right]$, if $\delta_n(g,g')<\delta$, then $\KantorovichDist{\Lip_n}(\alpha_n^g,\alpha_n^h) < \frac{\varepsilon}{2}$.
  
  Since  $\lim_{n\rightarrow\infty}r_n = \infty$, there exists $N_1\in\N$ such that, if $n\geq N_1$, then $r_n > \max\{\frac{1}{\delta},\frac{1}{\varepsilon}+\varepsilon\}$, and
  \begin{equation*}
    \sup_{g,g' \in G_n[r_n]}\sup_{h \in G_\infty[r_n]} \left|\delta_\infty(\xi_n(g)\xi_n(g'),h)-\delta_n(g g', f_n(h))\right|< \min\left\{\delta,\varepsilon\right\}\text,
  \end{equation*}
  and
  \begin{equation*}
    \sup_{g \in G_\infty[r_n]}\sup_{h,h' \in G_n[r_n]} \left|\delta_\infty(g,\xi_n(h')\xi_n(h))-\delta_n(f_n(g), h h')\right|< \min\left\{\delta,\varepsilon\right\}\text.
  \end{equation*}
  
  By \cite[Lemma 2.10]{Latremoliere18b}, we have $\delta_\infty(\xi_n(f_n(g)),g) < \delta$ for all $g \in G_\infty\left[\frac{1}{\varepsilon}\right]$  Therefore, for all $a\in\dom{\Lip_\infty}$ with $\Lip_\infty(a)\leq 1$,
  \begin{equation}\label{bridge-builder-eq-0}
    \norm{\alpha_\infty^g(a) - \alpha_{\infty}^{\xi_n(f_n(g))}(a)}{\A_\infty} \leq \KantorovichDist{\Lip_\infty}(\alpha_\infty^g,\alpha_\infty^{\xi_n(f_n(g))}) < \frac{\varepsilon}{2} \text,
  \end{equation}
 Moreover, since  $\delta_n(f_n(\xi_n(g)),g) < \delta$ for all $g \in G_n\left[\frac{1}{\varepsilon}\right]$, then for all $a\in\dom{\Lip_n}$ with $\Lip_n(a)\leq 1$,
  \begin{equation}\label{bridge-builder-eq-00}
    \norm{\alpha_n^g(a) - \alpha_n^{f_n(\xi_n(g))}(a)}{\A_\infty} \leq \KantorovichDist{\Lip_n}(\alpha_n^g,\alpha_n^{f_n(\xi_n(g))}) < \frac{\varepsilon}{2} \text.
  \end{equation}
  
  Moreover, by definition of covariant bridge builders, there exists $N_2 \in \N$ such that, if $n\geq N_2$, then:
  \begin{enumerate}
  \item for every $a\in\dom{\Lip_\infty}$, there exists $b \in \dom{\Lip_n}$ with $\Lip_n(b)\leq\Lip_\infty(a)$ and 
    \begin{equation}\label{bridge-builder-eq-1}
      \sup_{g \in G_\infty[r_n]} \norm{\pi(\alpha_\infty^g(a))-\alpha_n^{f_n(g)}(b)}{\A_\infty}  < \frac{\varepsilon}{2} \Lip_\infty(a)\text,
    \end{equation}
  \item for every $b\in\dom{\Lip_n}$, there exists $a \in \dom{\Lip_\infty}$ with $\Lip_\infty(a)\leq\Lip_n(b)$ and 
    \begin{equation}\label{bridge-builder-eq-2}
      \sup_{g \in G_n[r_n]} \norm{\pi(\alpha_\infty^{\xi_n(g)}(a))-\alpha_n^{g}(b)}{\A_\infty}  < \frac{\varepsilon}{2} \Lip_n(b) \text.
    \end{equation}
  \end{enumerate}	
  
  Fix $n\geq N\coloneqq\max\{N_0,N_1,N_2\}$. Let $\D_n \coloneqq \A_\infty \oplus \A_n$, and for all $a\in\dom{\Lip_\infty}$ and $b \in \dom{\Lip_n}$, set
  \begin{equation*}
    \TLip_n(a,b) \coloneqq \max\left\{ \Lip_\infty(a), \Lip_n(b), \frac{1}{ 2\, \varepsilon}\sup_{g \in G_\infty[\varepsilon^{-1}]}\norm{\pi\circ\alpha_\infty^{g}(a)-\alpha_n^{f_n(g)}(b)}{\A_\infty} \right\} \text.
  \end{equation*}
  Last, we set $p_n : (a,b)\in\D_n\mapsto a\in \A_\infty$ and $q_n:(a,b)\in\D_n\mapsto b \in \A_n$.

  We now follow the argument in \cite[Proposition 2.21]{FaLaPa}. First, we see that $p_n$ and $q_n$ are quantum isometries. By Expression (\ref{bridge-builder-eq-1}), if $a\in\dom{\Lip_\infty}$ with $\Lip_\infty(a) = 1$, then there exists $b\in\dom{\Lip_n}$ with $\Lip_n(b)\leq 1$ and $\sup_{g\in G_\infty[r_n]}\norm{\pi\circ\alpha_\infty^g(a)-\alpha_n^{f_n(g)}(b)}{\A_\infty}<\frac{\varepsilon}{2}$. Therefore, $\TLip_n(a,b) = 1 = \Lip_\infty(a)$. So $p_n$ is a quantum isometry.
  
  On the other hand, by Expression (\ref{bridge-builder-eq-2}), if $b \in \dom{\Lip_n}$ with $\Lip_n(b)=1$, then there exists $a\in \dom{\Lip_\infty}$ with $\Lip_\infty(a)\leq 1$ and \[\sup_{g\in G_n[r_n]}\norm{\pi(\alpha_\infty^{\xi_n(g)}(a)) - \alpha_n^{g}(b)}{\A_\infty} < \frac{\varepsilon}{2}.\] Henceforth, by Lemma (\ref{cov-bridge-lemma}), we obtain:
  \begin{equation*}
   \sup_{g\in G_\infty[r_n]} \norm{\pi(\alpha_\infty^{g}(a)) - \alpha_n^{f_n(g)}(b)}{\A_\infty} <  \varepsilon \text.
  \end{equation*}

  Therefore, $\TLip_n(a,b) = 1 = \Lip_n(b)$, so the $\ast$-morphism $q_n$ is a quantum isometry from $(\D_n,\TLip_n)$ to $(\A_\infty,\Lip_\infty)$.
  
  We refer to \cite[Theorem 2.22]{FaLaPa} for the proof that $\Haus{\TLip_n}(\StateSpace(\D_n),p_n^\ast(\StateSpace(\A_n))) \leq \frac{\varepsilon}{2}$ and $\Haus{\TLip_n}(\StateSpace(\D_n),q_n^\ast(\StateSpace(\A_n))) \leq \frac{\varepsilon}{2}$, which follows from the fact that if $(a,b)\in\dom{\TLip_n}$ with $\TLip_n(a,b)\leq 1$ then in particular, $\norm{\pi(a)-b}{\A_\infty} < \frac{\varepsilon}{2}$.
  
  \medskip
  
  Now fix $\varphi \in \StateSpace(\A_n)$ and let $g \in G_n[r_n]$. If $(a,b) \in \dom{\TLip_n}$ with $\TLip_n(a,b)\leq 1$, then $\norm{\pi\circ\alpha_\infty^g(a)-\alpha_n^{f_n(g)}(b)}{\A_\infty} < \varepsilon$. Consequently, if $\varphi'$ is a state of $\A_\infty$ given by extending $\varphi$ to $\A_\infty$ using the Hahn-Banach theorem, and if we set $\psi \coloneqq \varphi'\circ\pi$, then $\psi \in \StateSpace(\A_\infty)$, and
  \begin{align*}
    \left|\varphi\circ\alpha_n^{g}\circ q_n(a,b) - \psi\circ\alpha_\infty^{\xi_n(g)}\circ p_n(a,b)\right|
    &=\left|\varphi\circ\alpha_n^{g}(b) - \varphi'\circ\pi\circ\alpha_\infty^{\xi_n(g)}(a)\right| \\
    &=\left|\varphi'(\alpha_n^{g}(b) - \pi\circ\alpha_\infty^{\xi_n(g)}(a))\right| \\
    &\leq \norm{\alpha_n^{g}(b)-\pi(\alpha_\infty^{\xi_n(g)}(a))}{\A_\infty} \\
    &\leq \underbracket[1pt]{\norm{\alpha_n^g(b)-\alpha_n^{f_n(\xi_n(g))}}{\A_\infty}}_{<\frac{\varepsilon}{2}\text{ by Exp. (\ref{bridge-builder-eq-00})}} \\
    &\quad + \underbracket[1pt]{\norm{\alpha_n^{f_n(\xi_n(g))}(a) - \pi(\alpha_\infty^{\xi_n(g)}(a))}{\A_\infty}}_{<\frac{\varepsilon}{2}\text{ since }\TLip_n(a,b)\leq 1} \\
    &< \varepsilon \text.
  \end{align*}
  Thus, $\sup_{g \in G_n[r_n]}\Kantorovich{\TLip_n}(\varphi\circ\alpha^{g}\circ q_n,\psi\circ\alpha^{\xi_n(g)}\circ p_n) < \varepsilon$.
  
  Now fix  $\varphi \in \StateSpace(\A_\infty)$. Let $\psi = \varphi\circ\pi^{-1}$, restricted to $\A_n$: thus, $\psi \in \StateSpace(\A_n)$, and moreover, for $g\in  G_\infty[r_n]$:
  \begin{align*}
    \left|\varphi\circ\alpha_\infty^g\circ p_n(a,b) - \psi\circ\alpha_n^{f_n(g)}\circ q_n(a,b)\right|
    &=|\varphi(\alpha_\infty^g(a)-\pi^{-1}(\alpha_n^{f_n(g)}(b)))| \\
    &\leq \norm{\alpha_\infty^g(a) - \pi^{-1}(\alpha_n^{f_n(g)}(b))}{\A_\infty} \\
    &= \norm{\pi(\alpha_\infty^g(a)) - \alpha_n^{f_n(g)}(b)}{\A_\infty} \\
    &< \varepsilon \text.
  \end{align*}
  Hence $\sup_{g \in G_\infty[r_n]}\Kantorovich{\TLip_n}(\varphi\circ\alpha_\infty^g\circ p_n , \psi\circ\alpha_n^{f_n(g)}\circ q_n) < \varepsilon$.
  
  Thus our theorem is proven.
\end{proof}

\medskip

Covariant bridge builders are in fact necessary for convergence, under rather mild assumptions, so that they provide a rather satisfactory picture for covariant convergence of Lipschitz dynamical systems built on inductive sequences.

\begin{theorem}\label{covariant-bridge-builder-converse-thm}
  For each $n\in\Nbar \coloneqq \N\cup\{\infty\}$, let $(\A_n,\Lip_n)$ be a {\qcms}, such that $\A_\infty = \closure{\bigcup_{n\in\N} \A_n}$, where $(\A_n)_{n\in\N}$ is an increasing (for $\subseteq$) sequence of C*-subalgebras of $\A_\infty$, with the unit of $\A_\infty$ in $\A_0$. Moreover, for each $n\in\Nbar$, let $\alpha_n$ be an action of a proper metric group $(G_n,\delta_n)$ on $(\A_n,\Lip_n)$ by full quantum isometries, and for all $n\in\N$, let $(f_n,\xi_n)$ be a $\varepsilon_n$-iso-iso from $G_\infty$ to $G_n$ in the sense of Notation \eqref{not:ep-iso-iso}, with $\lim_{n\rightarrow\infty}\varepsilon_n = 0$. Assume in addition that  $G_n$ , $\alpha_n,$
  $f_n: G_\infty \to G_n $ and $\xi_n:G_n \to  G_\infty $  satisfy  Hypothesis (\ref{cov-bridge-hyp}).
  If
  \begin{equation*}
    \lim_{n\rightarrow\infty}\covpropinquity{}((\A_n,\Lip_n,G_n,\alpha_n),(\A_\infty,\Lip_\infty,G_\infty,\alpha_\infty)) = 0\text,
  \end{equation*}
  and if there exists $M > 0$ such that for all $n\in\N$
  \begin{equation*}
    \frac{1}{M}\Lip_n \leq \Lip_\infty \leq M \Lip_n \text{ on }\dom{\Lip_n},
  \end{equation*}
  then there exists  a covariant bridge builder associated to this data.
\end{theorem}

\begin{proof}
  Since $(\A_n,\Lip_n,G_n,\alpha_n)_{n\in\N}$ converges for the covariant propinquity to $(\A_\infty,\Lip_\infty,\allowbreak G_\infty,\alpha_\infty)$, there exists a sequence $(\varepsilon_n)_{n\in\N}$ in $(0,\infty)$ and a sequence $(\tau_n,f_n,\xi_n)_{n\in\N}$ where, for each $n\in\N$:
  \begin{itemize}
  \item $\lim_{n\rightarrow\infty}\varepsilon_n = 0$,
  \item $(f_n,\xi_n)$ is a $\varepsilon_n$-iso-iso,
  \item $\tau_n\coloneqq (\D_n,\TLip_n,\Theta_n,\theta_n)$ is a tunnel \cite{Latremoliere13b} from $(\A_\infty,\Lip_\infty)$ to $(\A_n,\Lip_n)$ with magnitude, with respect to $(f_n,\xi_n)$, no more than $\varepsilon_n$ in the sense of \cite[Definiton 3.11]{Latremoliere18b}; we now recall the meaning of these terms for our purpose. The pair $(\D_n,\TLip_n)$ is a {\qcms}, and there are two quantum isometries $\Theta_n:(\D_n,\TLip_n)\rightarrow(\A_\infty,\Lip_\infty)$ and $\theta_n : (\D_n,\TLip_n)\rightarrow(\A_n,\Lip_n)$, such that:
    \begin{equation}\label{eq:reach1}
      \sup_{\varphi\in\StateSpace(\A_\infty)}\inf_{\psi\in\StateSpace(\A_n)}\sup_{g\in G_\infty[r_n]}\Kantorovich{\TLip_n}(\varphi\circ\alpha_\infty^{g}\circ\Theta_n,\psi\circ\alpha_n^{f_n(g)}\circ\theta_n) < \varepsilon_n \text,
    \end{equation}
    and
    \begin{equation}\label{eq:reach2}
      \sup_{\psi\in\StateSpace(\A_n)}\inf_{\varphi\in\StateSpace(\A_\infty)}\sup_{g\in G_n[r_n]}\Kantorovich{\TLip_n}(\varphi\circ\alpha_\infty^{\xi_n(g)}\circ\Theta_n,\psi\circ\alpha_n^{g}\circ\theta_n) < \varepsilon_n \text.
    \end{equation}
  \end{itemize}
  
   We note that by  \cite[Definition 3.10]{Latremoliere18b}, the covariant reach of a tunnel $\tau_n$ is the largest  of the two suprema in Equations \eqref{eq:reach1} and \eqref{eq:reach2}.
   
  By \cite[Theorem 2.22]{FaLaPa}, convergence for the propinquity of our inductive sequence implies the existence of a bridge builder. Specifically, up to replacing $(\A_n,\Lip_n)_{n\in\N}$ by a subsequence, for all $a \in \dom{\Lip_\infty}$, we recall that $\{\pi(a)\}$ is the limit of $(\targetsettunnel{n}{a}{l})_{n\in\N}$ for $l\geq\Lip_\infty(a)$, in the sense of the Hausdorff distance induced on closed subsets of $\A_\infty$ by $\norm{\cdot}{\A_\infty}$, and where the target sets $\targetsettunnel{n}{a}{l}$ were defined in \cite{Latremoliere13} as:
  \begin{equation*}
    \targetsettunnel{n}{a}{l} \coloneqq \left\{ \Theta_n(d) : d \in \dom{\TLip_n}, \TLip_n(d)\leq l, \theta_n(d) = a \right\} \text.
  \end{equation*}

  Let $a\in\dom{\Lip_\infty}$ and let $\varepsilon > 0$. Then there exists  $N\in\N$ such that for all $n\geq N$, we have $\norm{\pi(a)-b}{\A_\infty} < \frac{\varepsilon}{6}$ for all $b \in \targetsettunnel{n}{a}{\Lip_\infty(a)}$.

  Let $g\in G_\infty[r_n]$ and let $\varphi \in \StateSpace(\A_\infty)$. Let $\varphi' \coloneqq \varphi_{|\A_n}$.  So  by Equations \eqref{eq:reach1} and \eqref{eq:reach2},
  there exists $\psi \in \StateSpace(\A_\infty)$ such that
  
   \begin{equation*}
  \sup_{g\in G_n[r_n]}\Kantorovich{\TLip_n}(\varphi\circ\alpha_\infty^{\xi_n(g)}\circ\Theta_n,\psi\circ\alpha_n^{g}\circ\theta_n) < \frac{\varepsilon}6 \text.
  \end{equation*}
  
  Let $a\in\dom{\Lip_\infty}$ with $\Lip_\infty(a)\leq 1$. Since $\Lip_\infty(a) =\Lip_\infty\circ\alpha_\infty^g(a)$, there exists $c_n \in \targetsettunnel{n}{\alpha_\infty^g(a)}{\Lip_\infty(a)}$. Let $b \in \targetsettunnel{n}{a}{\Lip_\infty(a)}$. We then compute, for $g \in G_n[r_n]$:
  \begin{align*}
    \left|\varphi(\pi(\alpha_\infty^{\xi_n(g)}(a)) - \alpha_n^{g}(b))\right|
    &\leq \left|\varphi(\pi(\alpha_\infty^{\xi_n(g)}(a)) - c_n)\right| + \underbracket[1pt]{|\varphi'(c_n-\alpha_n^{g}(b))|}_{c_n-\alpha_n^{g}(b) \in \A_n} \\
    &< \frac{\varepsilon}{6} + |\varphi'(c_n)-\psi(\alpha_\infty^{\xi_n(g)}(a))| + |\psi(\alpha_\infty^{\xi_n(g)}(a))-\varphi'(\alpha_n^{g}(b))| \\
    &< \frac{\varepsilon}{6} + \frac{\varepsilon}{6} + \frac{\varepsilon}{6} = \frac{\varepsilon}{2} \text.
  \end{align*}

  Since $\varphi \in \StateSpace(\A_\infty)$ was arbitrary, and since $\pi(\alpha_\infty^{\xi_n(g)}(a)) - \alpha_n^{g}(b) \in \sa{\A_\infty}$, we conclude that
  \begin{equation*}
    \sup_{g \in G_n[r_n]} \norm{\pi(\alpha_\infty^{\xi_n(g)}(a)) - \alpha_n^{g}(b)}{\A_\infty} < \frac{\varepsilon}{2} \text.
  \end{equation*}
  By Lemma (\ref{cov-bridge-lemma}), there exists $N\in\N$ such that, if $n\geq N$, then
  \begin{equation*}
    \sup_{g \in G_\infty[r_n]} \norm{\pi(\alpha_\infty^{g}(a)) - \alpha_n^{f_n(g)}(b)}{\A_\infty} < \varepsilon \text.
  \end{equation*}

  We now choose $b \in \dom{\Lip_n}$ with $\Lip_n(b) \leq 1$. Let $a\in\dom{\Lip_\infty}$ such that $b\in\targetsettunnel{n}{a}{\Lip_n(b)}$; by \cite[Theorem 2.22]{FaLaPa}, we have $\norm{\pi(a)-b}{\A_\infty} < \frac{\varepsilon}{3}$. So
  \begin{align*}
    \left|\varphi(\pi(\alpha_\infty^{\xi_n(g)}(a)) - \alpha_n^{g}(b))\right|
    &\leq \left|\varphi(\pi(\alpha_\infty^g(a)) - c_n)\right| + \underbracket[1pt]{|\varphi'(c_n-\alpha_n^{f_n(g)}(b)) |}_{c_n-\alpha_n^{f_n(g)}(b) \in \A_n} \\
    &< \frac{\varepsilon}{3} + |\varphi'(c_n)-\psi(\alpha_\infty^g(a))| + |\psi(\alpha_\infty^g(a))-\varphi'(\alpha_n^{f_n(g)}(b))| \\
    &< \frac{\varepsilon}{3} + \frac{\varepsilon}{3} + \frac{\varepsilon}{3} = \varepsilon \text.
  \end{align*}

  This completes our proof.
\end{proof}

\medskip

We now focus our attention to a situation which will fit many further examples. We narrow our Hypothesis (\ref{cov-bridge-hyp}) by requiring that our groups form a projective sequence, and  more specific form for our iso-isos.

\begin{hypothesis}\label{cov-bridge-proj-hyp}
For each $n\in\Nbar \coloneqq \N\cup\{\infty\}$, let $(\A_n,\Lip_n)$ be a {\qcms}, such that $\A_\infty = \closure{\bigcup_{n\in\N} \A_n}$, where $(\A_n)_{n\in\N}$ is an increasing (for $\subseteq$) sequence of C*-subalgebras of $\A_\infty$, with the unit of $\A_\infty$ in $\A_0$. 
  
  	Let
  \begin{equation*}
    G_0 \xleftarrow{j_0} G_1 \xleftarrow{j_1} G_2 \xleftarrow{j_2} G_3 \xleftarrow{j_3} \cdots
  \end{equation*}
  be a projective sequence of groups such that $j_m$ is surjective for all $m \in \N$, and let $G_\infty$ be its projective limit. Let $f_m : G_\infty \rightarrow G_m$ be the canonical surjective group morphism for all $m\in\N$, such that $j_{m-1} \circ f_m =  f_{m-1}$ for all $m\in\N\setminus\{0\}$. 
  
  Assume that, for each $n\in\Nbar$, the group $G_n$ is endowed with a proper left invariant metric $\delta_n$, and that we are given an action $\alpha_n$ of $G_n$ by full quantum isometries of $(\A_n,\Lip_n)$.

  Assume, moreover, that, for all $n\in\N$, there exists a unital map $\xi_n : G_n\rightarrow G_\infty$ and $r_n > 0$ such that  $\lim_{n\rightarrow\infty}r_n =\infty$ with
  \begin{equation*}
    \lim_{n\rightarrow\infty} \sup_{g,g' \in G_n[r_n]}\sup_{h \in G_\infty[r_n]} \left|\delta_\infty(\xi_n(g)\xi_n(g'),h)-\delta_n(g g', f_n(h))\right| = 0 \text.
  \end{equation*}
\end{hypothesis}

We begin with the observation that Hypothesis (\ref{cov-bridge-proj-hyp}) indeed provides $\varepsilon$-iso-iso maps, and thus include the convergence of the sequence of groups $(G_n)_{n\in\N}$ in the sense of the Gromov-Hausdorff distance for proper monoids.

\begin{lemma}\label{upsilon-lemma}
  Let $G$ and $H$ be two proper metric groups with respective units $e_G$ and $e_H$, and respective (left invariant) metrics $\delta_G$ and $\delta_H$. Assume that there exists a surjective group morphism $f : G\rightarrow H$. If $\xi : H \rightarrow G$ is a map with $\xi(e_H)=e_G$, and if there exist $r>0$ and $\varepsilon > 0$ with the property that:
  \begin{equation*}
    \sup_{h,h' \in H\left[r\right]} \sup_{g \in G\left[r\right]} \left|\delta_G(\xi(h)\xi(h'),g) - \delta_H(h h', f(g))\right| < \varepsilon \text,
  \end{equation*}
  then $(f,\xi)$ is a $\frac{r}{2}$-local $\varepsilon$-almost isometric isomorphism from $G$ to $H$, in the sense of Equations \eqref{rn-epsilon-iso-eq-1} and \eqref{rn-epsilon-iso-eq-2}.
\end{lemma}

\begin{proof}
  We already assumed half of the statement, so we only need to check the following.
  
  Let $h \in H\left[\frac{r}{2}\right]$ and $g,g' \in G\left[\frac{r}{2}\right]$. Since $\delta_G$ is left-invariant, we then get
  \begin{equation*}
    \delta_G(g g',e_G) \leq \delta_G(g g',g) + \delta_G(g,e_G) = \delta_G(g',e_G) + \delta_G(g,e_G) \leq \frac{r}{2} + \frac{r}{2} = r \text.
  \end{equation*}
  Therefore, setting $z \coloneqq g g' \in G\left[\frac{r}{2}\right]$ for clarity, by assumption on $f$ and $\xi$, we see that
  \begin{align*}
    \left| \delta_G(g g',\xi(h)) - \delta_H(f(g)f(g'),h)\right|
    &=\left| \delta_G(g g',\xi(h)) - \delta_H(\underbracket[1pt]{f(g g')}_{f \text{group morphism}},h) \right| \\
    &= \left| \delta_G( z, \xi(h)\underbracket[1pt]{\xi(e_{H})}_{=e_{G}} ) - \delta_H(f (z), h e_{H}) \right| \\
    &< \varepsilon \text{ by assumption on $\xi$ .}
  \end{align*}

  This concludes our proof.
\end{proof}

\begin{corollary}\label{cor:Y-conv}
  If we assume Hypothesis (\ref{cov-bridge-proj-hyp}), then
  \begin{equation*}
    \lim_{n\rightarrow\infty} \Upsilon((G_n,\delta_n),(G_\infty,\delta_\infty)) = 0 \text,
  \end{equation*}
  where we denote by $\Upsilon$ the {\it Gromov-Hausdorff monoid distance}, see \cite[Definition 2.7]{Latremoliere18b}.
\end{corollary}

\begin{proof}
  Let $\varepsilon > 0$. Then there exists $N \in \N$ such that if $n\geq N$, then $r_n \geq \frac{2}{\varepsilon} + \varepsilon$, and
  \begin{equation*}
    \sup_{g,g' \in G_n[r_n]}\sup_{h \in G_\infty[r_n]} \left|\delta_\infty(\xi_n(g)\xi_n(g'),h)-\delta_n(g g', f_n(h))\right| < \varepsilon \text.
  \end{equation*}
  Thus, by Lemma (\ref{upsilon-lemma}), $(f_n,\xi_n)$ is a $\frac{1}{\varepsilon}$-local $\varepsilon$-almost isometric isomorphism. So by \cite[Definition 2.7]{Latremoliere18b}, we conclude $\Upsilon((G_n,\delta_n),(G_\infty,\delta_\infty)) < \varepsilon$. This concludes the proof of Corollary \eqref{cor:Y-conv}.
\end{proof}

\begin{corollary}\label{bridge-builder-commute-cor}
  Under the same assumptions as in Theorem (\ref{covariant-bridge-builder-thm}), if there exists a bridge builder $\pi$ such that
  \begin{enumerate}
  \item for each $n\in\Nbar$, for all $g \in G_n$, the $\ast$-automorphism $\alpha_n^g$ commute with $\pi$,
  \item for all $g \in G_\infty$, and for all $n\in\N$, the restriction of the $\ast$-automorphism  $\alpha_\infty^{g}$ to $\A_n$ equals $\alpha_n^{f_n(g)}$,
  \end{enumerate}
  then
  \begin{equation*}
    \lim_{n\rightarrow\infty} \covpropinquity{}((\A_n,\Lip_n,\alpha_n,G_n),(\A_\infty,\Lip_\infty,\alpha_\infty,G_\infty)) = 0 \text.
  \end{equation*}  
\end{corollary}

\begin{proof}
	Let $\varepsilon > 0$; by definition of $\pi$, there exists $N\in\N$ such that for all $n\in\N$, Expressions (\ref{bridge-builder-eq-1}) and (\ref{bridge-builder-eq-2}) hold for $\frac{\varepsilon}{2}$ instead of $\varepsilon$. Let $\delta>0$ and $N_0 \in \Nbar$ such that for all $n\geq N_0$, if $g,g'\in G_n$ with $\delta_n(g,g')<\delta$, then 
	\begin{equation}\label{bridge-builder-commute-cor-eq-1-mkD}
	\KantorovichDist{\Lip_n}(\alpha_n^g,\alpha_n^{g'}) < \varepsilon \text.
	\end{equation}
	
	Moreover, by Hypothesis (\ref{cov-bridge-hyp}) and Lemma (\ref{upsilon-lemma}), there exists $N_1\in\N$ such that, for all $n\geq N_1$, we have $r_n > \frac{1}{\varepsilon} + \varepsilon$, and for all $g,g'\in G_\infty\left[\frac{1}{\varepsilon}\right]$ and for all $h \in G_n\left[\frac{1}{\varepsilon}\right]$,
	\begin{equation*}
	\left| \delta_n(f_n(gg'),h) - \delta_\infty(gg',\xi_n(h))\right| < \varepsilon \text.
	\end{equation*}

	Let $a\in \dom{\Lip_\infty}$. Since $\pi$ is a bridge builder, there exists $b \in \dom{\Lip_n}$ with $\Lip_n(b)\leq\Lip_\infty(a)$ such that $\norm{\pi(a)-b}{\A_\infty} < \frac{\varepsilon}{2} \Lip_\infty(a)$. Then for all $g \in G_\infty[r_n]$,
	\begin{align*}
	\norm{\pi(\alpha_\infty^{g}(a)) - \alpha_n^{f_n(g)}(b)}{\A_\infty}
	&=\underbracket[1pt]{\norm{\alpha_\infty^{g}(\pi(a)) - \alpha_n^{f_n(g)}(b)}{\A_\infty}}_{[\pi,\alpha^g]=0} \\
	&\leq \norm{\alpha_\infty^g(\pi(a)-b)}{\A_\infty} + \norm{\alpha_\infty^g(b)-\alpha_n^{f_n(g)}(b)}{\A_\infty} \\
	&\leq \norm{\pi(a)-b}{\A_\infty} + 0 \\
	&< \varepsilon \Lip_\infty(a) \text.
	\end{align*}
	Thus, we have shown that, for all $a\in \dom{\Lip_\infty}$, there exists $b \in \dom{\Lip_n}$ with $\Lip_n(b)\leq\Lip_\infty(a)$ and
	\begin{equation*}
	\sup_{g\in G_\infty[r_n]}\norm{\pi(\alpha^g(a)) - \alpha^{f_n(g)}(a)}{\A_\infty} < \varepsilon \Lip_\infty(a)\text.
	\end{equation*}
	
	Now fix $b \in \dom{\Lip_n}$. Since $\pi$ is a bridge builder, with our choice of $n\geq N$, there exists $a \in \dom{\Lip_\infty}$ with $\Lip_\infty(a)\leq \Lip_n(b)$ and $\norm{\pi(a)-b}{\A_\infty}<\frac{\varepsilon}{2} \Lip_n(b)$. For all $g \in G_n[r_n]$, since $f_n$ is surjective, there exists $h \in G_\infty$ such that $g=f_n(h)$. Since $\alpha_n^g$ and $\pi$ commute, we obtain:
	\begin{align*}
	\norm{\pi(\alpha_\infty^{\xi_n(g)}(a))-\alpha_n^g(b)}{\A_\infty}
	&\leq \norm{\alpha_\infty^{\xi_n(g)}(\pi(a)-b)}{\A_\infty} + \norm{\alpha_\infty^{\xi_n(g)}(b) - \alpha_n^g(b)}{\A_\infty} \\
	&= \norm{\pi(a)-b}{\A_\infty} + \norm{\alpha_\infty^{\xi_n(f_n(h))}(b) - \alpha_n^{f_n(h)}(b)}{\A_\infty} \\
	&\leq \norm{\pi(a)-b}{\A_\infty} + \underbracket[1pt]{\norm{\alpha_\infty^{\xi_n(f_n(h))}(b) - \alpha_\infty^h(b)}{\A_\infty}}_{\text{ by Eq. \eqref{bridge-builder-commute-cor-eq-1-mkD}}  } \\
	&\quad + \underbracket[1pt]{\norm{\alpha_\infty^{h}(b) - \alpha_n^{f_n(h)}(b)}{\A_\infty}}_{=0\text{ by (2)}} \\
	&= \norm{\pi(a)-b}{\A_\infty} + \frac{\varepsilon}{2} \Lip_n(b) + 0 \\
	&< \varepsilon \Lip_n(b) \text.
	\end{align*}
	
	Hence, $\pi$ is actually a covariant bridge builder, and we meet the hypothesis of Theorem (\ref{covariant-bridge-builder-thm}), so our corollary follows.
\end{proof}

We now focus our attention on groups of quantum isometries. If $G\subseteq \Aut{\A}$ is a closed subgroup of $\Aut{\A}$ for the pointwise convergence topology, where $(\A,\Lip)$ is a {\qcms}, we will simply write $(\A,\Lip,G)$ for the associated Lipschitz dynamical system induced by the obvious action of $G$ on $\A$. In this case, we can simplify the situation further, thanks to the following theorem. The main lesson from this theorem is that we can obtain convergence of various groups of quantum isometries, and of their actions, as long as we can extend quantum isometries from a C*-subalgebra in the inductive limit, to the entire limit.

\begin{definition}
 For each $n\in\Nbar \coloneqq \N\cup\{\infty\}$, let $(\A_n,\Lip_n)$ be a {\qcms}, such that $\A_\infty = \closure{\bigcup_{n\in\N} \A_n}$, where $(\A_n)_{n\in\N}$ is an increasing (for $\subseteq$) sequence of C*-subalgebras of $\A_\infty$, with the unit of $\A_\infty$ in $\A_0$. If $\pi$ is any $\ast$-automorphism of $\A_\infty$, then define:
	\begin{equation*}
	\LargeIso(\A_\infty,\Lip_\infty|\pi) \coloneqq \left\{ \alpha\in \LargeIso(\A_\infty,\Lip_\infty) : \pi\circ\alpha=\alpha\circ\pi\text{ and, }\forall n \in \N \quad \alpha_{|\A_n} \in \LargeIso(\A_n,\Lip_n) \right\} \text.
	\end{equation*}
\end{definition}

\begin{theorem}\label{thm:iso-conv}
	For each $n\in\Nbar \coloneqq \N\cup\{\infty\}$, let $(\A_n,\Lip_n)$ be a {\qcms}, such that $\A_\infty = \closure{\bigcup_{n\in\N} \A_n}$, where $(\A_n)_{n\in\N}$ is an increasing (for $\subseteq$) sequence of C*-subalgebras of $\A_\infty$, with the unit of $\A_\infty$ in $\A_0$. Let $\pi$ be a full quantum isometry which is also a bridge builder. 
	The group $\LargeIso(\A_\infty,\Lip_\infty|\pi)$ is a compact subgroup of the $\ast$-automorphism group $\Aut{\A_\infty}$.
	
	Furthermore, if $G_\infty$ is a closed subgroup of $\LargeIso(\A_\infty,\Lip_\infty|\pi)$ and if, for all $n\in\N$, we set $G_n \coloneqq \left\{ \alpha_{|\A_n} \in \Aut{\A_n} : \alpha\in G_\infty \right\}$, then $G_n$ is a closed group, and
	\begin{equation*}
	\lim_{n\rightarrow\infty} \covpropinquity{}((\A_n,\Lip_n,G_n),(\A_\infty,\Lip_\infty,G_\infty)) = 0 \text.
	\end{equation*}
\end{theorem}

\begin{proof}
  If $(\alpha_n)_{n\in\N}$ is a sequence in $\LargeIso(\A_\infty,\Lip_\infty|\pi)$ converging in $\Aut{\A_\infty}$ to some $\alpha$ for $\KantorovichDist{\Lip_\infty}$, then, for all $a\in\A_n$,  the sequence $(\alpha_n(a))_{n\in\N}$ is norm convergent, and thus $\alpha_\infty(a) \in \A_n$ as $\A_n$ is closed in $\A_\infty$; moreover by continuity, $\alpha(\pi(a)) = \pi(\alpha(a))$ for all $a\in\A_\infty$. Therefore, $\alpha\in\LargeIso(\A_\infty,\Lip_\infty|\pi)$. So $\LargeIso(\A_\infty,\Lip_\infty|\pi)$ is a closed subset of $\Aut{\A_\infty}$. Since $\LargeIso(\A_\infty,\Lip_\infty)$ is a compact metric group by Theorem (\ref{compact-iso-group-thm}), it then follows that $\LargeIso(\A_\infty,\Lip_\infty|\pi)$ is a compact subgroup of $\LargeIso(\A_\infty,\Lip_\infty)$.
  
  \medskip
  
  Fix $n\in\N$. For $\alpha \in G_\infty$, we simply write $f_n(\alpha)$ for the restriction $\alpha_{|\A_n}$ of $\alpha$ to $\A_n$, so $f_n(\alpha) \in G_n$. Thus, $f_n$ is a surjective group morphism. On the other hand, if $\beta\in G_n$, then by definition, there exists $\alpha\in G_\infty$ such that $\beta=f_n(\alpha)$; we set $\xi_n(\beta)\coloneqq \alpha$ for such a choice of $\alpha$ (this choice is by no means unique in general).

By construction, the topology induced by $\KantorovichDist{\Lip_\infty}$ (resp. $\KantorovichDist{\Lip_n}$) on $G_\infty$ (resp. $G_n$) is the topology of pointwise convergence, and $f_n$ is continuous for these topologies. As $G_n = f_n(G_\infty)$, the group $G_n$ is also compact since the continuous image of the compact set $G_\infty$ (indeed $G_\infty$ is compact since it is  a closed subset of the compact group $\LargeIso(\A_\infty,\Lip_\infty)$, cf. Theorem (\ref{compact-iso-group-thm})). Hence $G_n$ is compact, and therefore closed as the topology on $G_n$ is metrized and so Hausdorff. 
  

	We now check that $(f_n,\xi_n)$ is a local almost isometric isomorphism for all $n\in\N$, as in \cite[Definition 2.5]{Latremoliere18b}. 
	
	Fix $\varepsilon > 0$ and let $N\in\N$ be given by the bridge builder property. Let $\beta,\beta' \in G_n$ and $\alpha \in G_\infty$. Now, let $a\in\dom{\Lip_\infty}$ with $\Lip_\infty(a) \leq 1$. Since $\pi$ is a bridge builder, there exists $b \in \A_n$ with $\Lip_n(b)\leq 1$, and $\norm{\pi(a)-b}{\A_\infty} < \frac{\varepsilon}{2}$. So
	\begin{align*}
	\norm{\xi_n(\beta)\circ\xi_n(\beta')(a) - \alpha(a)}{\A_\infty}
	&=\underbracket[1pt]{\norm{\pi\left(\xi_n(\beta)\circ\xi_n(\beta')(a) - \alpha(a)\right)}{\A_\infty}}_{\text{since $\pi$ is an isometry for $\norm{\cdot}{\A_\infty}$}} \\
	&\leq \norm{\pi\left(\xi_n(\beta)\circ\xi_n(\beta')(a)-\xi_n(\beta)\circ\xi_n(\beta')(\pi^{-1}(b))\right)}{\A_\infty} \\
	&\quad + \norm{\pi\left(\xi_n(\beta)\circ\xi_n(\beta')(\pi^{-1}(b)) - \alpha(a)\right)}{\A_\infty} \\
	&\leq \underbracket[1pt]{\norm{\xi_n(\beta)\circ\xi_n(\beta')(\pi(a) - b)}{\A_\infty}}_{\text{ since $\pi$ commutes with elements in $G_\infty$}} \\
	&\quad + \underbracket[1pt]{\norm{\xi_n(\beta)\circ\xi_n(\beta')(b) - \alpha(\pi(a))}{\A_\infty}}_{\text{ since $\pi$ commutes with elements in $G_\infty$}} \\
	&\leq \underbracket[1pt]{\norm{\pi(a)-b}{\A_\infty}}_{\xi_n(\beta)\circ\xi_n(\beta')\in\Aut{\A_\infty}} + \underbracket[1pt]{\norm{\beta\circ\beta'(b) - \alpha(\pi(a))}{\A_\infty}}_{b \in \A_n} \\
	&< \frac{\varepsilon}{2} + \norm{\beta\circ\beta'(b) - \alpha(b)}{\A_\infty} + \norm{\alpha(b)-\alpha(\pi(a))}{\A_\infty} \\
	&\leq \frac{\varepsilon}{2} + \underbracket[1pt]{\norm{\beta\circ\beta'(b) - f_n(\alpha)(b)}{\A_\infty}}_{\text{by def. of }f_n} + \underbracket[1pt]{\norm{b-\pi(a)}{\A_\infty}}_{\alpha\in\Aut{\A_\infty}} \\
	&< \frac{\varepsilon}{2} + \KantorovichDist{\Lip_n}(\beta\circ\beta',f_n(\alpha)) + \frac{\varepsilon}{2} \\
	&= \KantorovichDist{\Lip_n}(\beta\circ\beta',f_n(\alpha)) + \varepsilon \text.
	\end{align*}
	Hence
	\begin{equation*}
	\KantorovichDist{\Lip_\infty}(\xi_n(\beta)\circ\xi_n(\beta'),\alpha) < \KantorovichDist{\Lip_n}(\beta\circ\beta',f_n(\alpha)) + \varepsilon\text.
	\end{equation*}

	On the other hand, let $b \in \dom{\Lip_n}$ with $\Lip_n(b)\leq 1$. Again, as $\pi$ is a bridge builder, there exists $a \in \dom{\Lip_\infty}$ with $\Lip_\infty(a)\leq 1$ and $\norm{\pi(a)-b}{\A_\infty} < \frac{\varepsilon}2$. Moreover by assumption, $\Lip_\infty(\pi(a)) \leq \Lip_\infty(a) \leq 1$. Therefore, we conclude:
	\begin{align*}
	\norm{\beta\circ\beta'(b)-f_n(\alpha)(b)}{\A_\infty}
	&\leq \norm{\beta\circ\beta'(b)-\xi_n(\beta)\circ\xi_n(\beta')(\pi(a))}{\A_\infty} \\
	&\quad + \norm{\xi_n(\beta)\circ\xi_n(\beta')(\pi(a)) - \alpha(\pi(a))}{\A_\infty} + \norm{\alpha(\pi(a))-\alpha(b)}{\A_\infty} \\
	&\leq \norm{\xi_n(\beta)\circ\xi_n(\beta')(b-\pi(a))}{\A_\infty} + \underbracket[1pt]{\KantorovichDist{\Lip_\infty}(\xi_n(\beta)\circ\xi_n(\beta'),\alpha)}_{\Lip_\infty(\pi(a))\leq 1} \\
	&\quad + \underbracket[1pt]{\norm{\pi(a)-b}{\A_\infty}}_{\alpha\in\Aut{\A}} \\
	&\leq \underbracket[1pt]{\norm{\pi(a)-b}{\A_\infty}}_{\xi_n(\beta)\circ\xi_n(\beta') \in \Aut{\A_\infty}} + \KantorovichDist{\Lip_\infty}(\xi_n(\beta)\circ\xi_n(\beta'),\alpha) + \frac{\varepsilon}{2} \\
	&< \frac{\varepsilon}{2} + \KantorovichDist{\Lip_\infty}(\xi_n(\beta)\circ\xi_n(\beta'),\alpha) + \frac{\varepsilon}{2} \\
	&= \KantorovichDist{\Lip_\infty}(\xi_n(\beta)\circ\xi_n(\beta'),\alpha) +  \varepsilon \text.
	\end{align*}
	Hence
	\begin{equation*}
	\KantorovichDist{\Lip_n}(\beta\circ\beta',f_n(\alpha)) < \KantorovichDist{\Lip_\infty}(\xi_n(\beta)\circ\xi_n(\beta),\beta)+ \varepsilon  \text.
	\end{equation*}
	
	Therefore
	\begin{equation*}
	\left| \KantorovichDist{\Lip_n}(\beta\circ\beta',f_n(\alpha))-\KantorovichDist{\Lip_\infty}(\xi_n(\beta)\circ\xi_n(\beta'),\alpha) \right| < \varepsilon \text.
	\end{equation*}
	
	\medskip
	
	Now, let $\alpha,\alpha' \in G_\infty$ and $\beta \in G_n$. Fix $b\in\dom{\Lip_n}$ with $\Lip_n(b)\leq 1$. Since $b \in \A_n$, the map $f_n$ is the restriction to $\A_n$, and $\xi_n$ extends automorphisms of $\A_n$, we conclude:
	\begin{equation*}
	\norm{f_n(\alpha\circ\alpha')(b) - \beta(b)}{\A_\infty} = \norm{\alpha\circ\alpha'(b) - \xi_n(\beta)(b)}{\A_\infty} \leq \KantorovichDist{\Lip_\infty}(\alpha\circ\alpha',\xi_n(\beta)) \text,
	\end{equation*}
	so $\KantorovichDist{\Lip_n}(f_n(\alpha\circ\alpha'),\beta) \leq \KantorovichDist{\Lip_\infty}(\alpha\circ\alpha',\xi_n(\beta))$.

	Lastly, if $a \in \dom{\Lip_\infty}$ with $\Lip_\infty(a)\leq 1$, there exists   $b \in \dom{\Lip_n}$ with $\Lip_n(b)\leq \Lip_\infty(a)\leq 1$ and $\norm{\pi(a)-b}{\A_\infty} < \frac{\varepsilon}2$. We then have:
	\begin{align*}
	\norm{\alpha\circ\alpha'(a) - \xi_n(\beta)(a)}{\A_\infty}
	&= \underbracket[1pt]{\norm{\pi\left(\alpha\circ\alpha'(a) - \xi_n(\beta)(a)\right)}{\A_\infty}}_{\pi\in\Aut{\A_\infty}} \\
	&\leq \underbracket[1pt]{\norm{\alpha\circ\alpha'(\pi(a)) - \alpha\circ\alpha'(b)}{\A_\infty}}_{\pi \text{ commutes with }\alpha\circ\alpha'} + \norm{\alpha\circ\alpha'(b)-\beta(b)}{\A_\infty} \\
	&\quad + \norm{\beta(b) - \xi_n(\beta)(\pi(a))}{\A_\infty} \\
	&\leq \underbracket[1pt]{\norm{\pi(a) - b}{\A_\infty}}_{\alpha\circ\alpha'\in\Aut{\A}} + \KantorovichDist{\Lip_n}(f_n(\alpha\circ\alpha'),\beta) + \norm{b-\pi(a)}{\A_\infty} \\
	&< \KantorovichDist{\Lip_n}(f_n(\alpha\circ\alpha'),\beta) + \varepsilon \text.
	\end{align*}
	
	We thus have shown that $\KantorovichDist{\Lip_\infty}(\alpha\circ\alpha',\xi_n(\beta)) \leq \KantorovichDist{\Lip_n}(f_n(\alpha\circ\alpha'),\beta) + \varepsilon$. Henceforth,
	\begin{equation*}
	\left| \KantorovichDist{\Lip_\infty}(\alpha\circ\alpha',\xi_n(\beta)) - \KantorovichDist{\Lip_n}(f_n(\alpha\circ\alpha'),\beta)\right| < \varepsilon \text.
	\end{equation*}
	
	Our result now follows from Corollary (\ref{bridge-builder-commute-cor}) by construction.
\end{proof}

\section{Applications}\label{sec:applications}

We now apply our work on covariant bridge builders to obtain metric convergence of  groups of isometries to their projective limits for certain inductive sequences of {\qcms s}.

\subsection{The Identity as a Bridge Builder}

A particular situation for the application of Theorem (\ref{thm:iso-conv}) is when the identity itself is a covariant bridge builder --- which indeed happens in several examples included later in this section. The commutation requirement of Theorem (\ref{thm:iso-conv}) is then automatically satisfied. We thus obtain the following nice convergence results, where the main problem left is whether we can extend quantum isometries from a subalgebra to the entire algebra. We phrase this by working with groups of quantum isometries of the inductive limit consisting of automorphisms which restrict to quantum isometries on each factor of the inductive sequence. Our first result in this section is about general quantum isometries and it is an immediate  consequence of Theorem \eqref{thm:iso-conv}.

\begin{corollary}\label{id-cov-cor}
  For each $n\in\Nbar \coloneqq \N\cup\{\infty\}$, let $(\A_n,\Lip_n)$ be a {\qcms}, such that $\A_\infty = \closure{\bigcup_{n\in\N} \A_n}$, where $(\A_n)_{n\in\N}$ is an increasing (for $\subseteq$) sequence of C*-subalgebras of $\A_\infty$, with the unit of $\A_\infty$ in $\A_0$. Assume that the identity automorphism of  $\A_\infty$ is also a bridge builder. 
  Define:
  
  \begin{equation*}
    G_\infty \coloneqq \left\{ \alpha\in \LargeIso(\A_\infty,\Lip_\infty) : \forall n\in \N  \quad \alpha_{|\A_n}\in\LargeIso(\A_n,\Lip_n) \right\}.
  \end{equation*}
  If, for each $n\in\N$, we define:
  \begin{equation*}
    G_n \coloneqq \left\{ \alpha_{|\A_n} \in \Aut{\A_n} : \alpha \in G_\infty \right\} \text,
  \end{equation*}
  then
  \begin{equation*}
    \lim_{n\rightarrow\infty} \covpropinquity{}((\A_\infty,\Lip_\infty,G_\infty)),(\A_n,\Lip_n,G_n)) = 0 \text.
  \end{equation*}
\end{corollary}

We now can look more closely at the group of quantum isometries of a metric spectral triple. Of course, these are special cases of compact subgroups of the group of quantum isometries of the induced {\qcms s}, but we want to find some reasonable condition to ensure that our convergence entirely takes place within the spectral triple context.

The following result is also a corollary of Theorem \eqref{thm:iso-conv}.

\begin{corollary}\label{id-cov-small-cor}
  Let $\A_\infty \coloneqq \closure{\bigcup_{n\in\N} \A_n}$, where $(\A_n)_{n\in\N}$ is an increasing (for $\subseteq$) sequence of C*-subalgebras of $\A_\infty$, with the unit of $\A_\infty$ in $\A_0$.  Assume that the identity automorphism of $\A_\infty$ is also a bridge builder.  Let $(\A_\infty,\Hilbert_\infty,\Dirac_\infty)$ be any metric  spectral triple on $\A_\infty$.
  Assume that $(\Hilbert_n)$ is an increasing  sequence of Hilbert subspaces of $\Hilbert_\infty$ with the properties:
\begin{enumerate}
\item For all $n \in \N$: $\A_n \Hilbert_n \subseteq \Hilbert_n$,
\item $\Hilbert_\infty = \closure{\bigcup_{n \in \N}\Hilbert_n}$,
\item $\Dirac_\infty$ commutes with the projections on $\Hilbert_n,$ for every $n \in \N$,
\item $(\A_n, \Hilbert_n, \Dirac_n)$ is metric for every $n \in \N$.
\end{enumerate}

Let:

\begin{equation*}
  J_\infty \coloneqq \left\{ \alpha\in \smallIso(\A_\infty,\Hilbert_\infty,\Dirac_\infty) : \forall n\in \N  \quad \alpha_{|\A_n}\in\LargeIso(\A_n,\Lip_n) \right\}
\end{equation*}
and
\begin{equation*}
  \begin{split}
    G_\infty \coloneqq \left\{ \alpha=\AdRep{U} \in {\Aut{\A_\infty}}: \ U\in \mathcal{U}(\Hilbert_\infty), [\Dirac_\infty, U]=0\text{ with the property: } \right. \\
    \left.  \forall n\in \N  \quad U\Hilbert_n \subseteq \Hilbert_n \text{ and }  \alpha(\A_n) \subseteq \A_n \right\},
  \end{split}
\end{equation*}
where $\mathcal{U}(\Hilbert_\infty)$ is the set of unitaries on $\Hilbert_\infty.$

If we define, for each $n\in\N$:
\begin{equation*}
  J_n \coloneqq \left\{ \alpha_{|\A_n} \in \Aut{\A_n} : \alpha \in J_\infty \right\} \text{ and }G_n \coloneqq \left\{ \alpha_{|\A_n} \in \Aut{\A_n} : \alpha \in G_\infty \right\} \text,
\end{equation*}
then,  for all $n\in\Nbar$:
\[G_n \subseteq \smallIso(\A_n,\Hilbert_n,\Dirac_n)\]
and

\begin{equation}\label{eq:J-eq}
  \lim_{n\rightarrow\infty} \covpropinquity{}((\A_\infty,\Lip_\infty,J_\infty),(\A_n,\Lip_n,J_n)) = 0 \text,
\end{equation}

\begin{equation}\label{eq:G-eq}
  \lim_{n\rightarrow\infty} \covpropinquity{}((\A_\infty,\Lip_\infty,G_\infty),(\A_n,\Lip_n,G_n)) = 0 \text,
\end{equation}

\end{corollary}

\begin{proof} Equation \eqref{eq:J-eq} follows immediately from Theorem \eqref{thm:iso-conv} since the group $J_\infty$ is closed by construction.
  
  To prove Equation \eqref{eq:G-eq}  first note that $G_\infty$ is a subgroup of $\smallIso(\A_\infty, \Hilbert_\infty, \Dirac_\infty  )$ by definition. Next, $G_\infty$ is closed in $\KantorovichDist{\Dirac_\infty}$. To prove that, let $\alpha_n$ be a sequence in $G_\infty$ which converges to $\alpha $ for $\KantorovichDist{\Dirac_\infty}$. Firstly, notice that, for all $k,n \in \N $ we have  $\alpha_n(\A_k) \subseteq \A_k, $ which implies $\alpha(\A_k) \subseteq \A_k $ since $\A_k$ is closed. Now, for all $n \in \N$, there exists a unitary $U_n$  on $\Hilbert_\infty$ such that $\alpha_n = \AdRep{U_n}$ with $[U_n, \Dirac_\infty]=0$ and $U_n \Hilbert_k \subseteq \Hilbert_k$ for all $k\in \N$. By the proof of Theorem \eqref{compact-iso-group-thm}, there exists a subsequence of $U_n$ which converges in the SOT to some unitary $U$ on $\Hilbert_\infty$ such that $\alpha= \AdRep{U}$ and $[\Dirac_\infty, U]=0$. Since $\Hilbert_k$ is closed for all $k
  \in \N$, then $U\Hilbert_k \subseteq \Hilbert_k.$ So $\alpha \in G_\infty. $ By applying Theorem \eqref{thm:iso-conv} we are done. 
\end{proof}

\subsection{Almost Inners}

A key observation from the previous section is that our convergence results involve quantum isometries on factors of an inductive limit which can somehow be extended to the inductive limit as quantum isometries.

In general, this is of course a delicate problem: automorphisms do not typically extend from subalgebras to the entire algebra. There are however simple cases where such extensions are almost possible: inner automorphisms obviously extend.

We will now see that the quantum metric structure allows to find other automorphisms which extend from subalgebras to the entire algebra.

\begin{theorem}\label{extending-auto-thm}
  Let $(\A,\Lip)$ be a {\qcms}, and let $\B\subseteq\A$ be a C*-subalgebra of $\A$ with $1\in \B$. Let $K > 0$. If $\beta \in \Aut{\B}$ is a $\ast$-automorphism of $\B$ which is the pointwise limit of $\ast$-automorphisms of $\B$, each of which extends to $\A$ as a $K$-bi-Lipschitz morphism, then there exists a $K$-bi-Lipschitz $\ast$-automorphism $\alpha \in \Aut{\A}$ whose restriction to $\B$ is $\beta$.
\end{theorem}

\begin{proof}
  Let $(\beta_n)_{n\in\N}$ be a sequence of $\ast$-automorphisms of $\B$ converging pointwise to $\beta$. By assumption, for each $n\in\N$, there exists a $K$-bi-Lipschitz automorphism $\alpha_n$ of $(\A,\Lip)$ whose restriction to $\B$ is $\alpha_n$. By Theorem (\ref{compact-Iso-thm}), the set of $K$-bi-Lipschitz automorphisms of $(\A,\Lip)$ is compact for the topology of pointwise convergence, so there exists a subsequence $(\alpha_{f(n)})_{n\in\N}$ which converges pointwise to a $K$-bi-Lipschitz automorphism $\alpha$ of $\A$. By construction, $\alpha$ restricts to $\beta$ on $\B$.
\end{proof}

We can apply Theorem (\ref{extending-auto-thm}) to inner automorphism quantum isometries. First, we make a quick digression about extending almost inner automorphisms.

\begin{corollary}
  Let $(\A,\Lip)$ be a {\qcms}, and let $\B\subseteq\A$ be a C*-subalgebra of $\A$ with $1\in \B$. Let $K > 0$. We define the set
  \begin{equation*}
    \mathscr{U}(\B|\A,\Lip,K) \coloneqq \{ u \in \mathcal{U}(\B) : \forall a \in \dom{\Lip} \quad \Lip(u a u^\ast)\leq K \Lip(a) \} \text.
  \end{equation*}

  If $\beta \in\Aut{\B}$ is the pointwise limit of inner automorphisms implemented as $\AdRep{u}$ for $u \in \mathscr{U}(\B)$, then there exists a $K$-bi-Lipschitz automorphism of $\A$ whose restriction to $\B$ is $\beta$.
\end{corollary}

\begin{proof}
  By assumption, every inner automorphism of the form $\AdRep{u}$ for $u \in \mathscr{U}(\B|\A,\Lip,K)$ extends to $\A$ (as it is inner) to a $K$-bi-Lipschitz automorphism. We then apply Theorem (\ref{extending-auto-thm}).
\end{proof}

\begin{remark}
  Assume that $\Lip$ is a seminorm defined on a dense subalgebra of the unital C*-algebra $\A$, such that $(\A,\Lip_{|\sa{\A}})$ is a {\qcms}, and for all $a,b \in \dom{\Lip}$, we have $\Lip(ab) \leq \Lip(a)\norm{b}{\A} + \norm{a}{\A} \Lip(b)$ --- for instance, $\Lip(a)=\opnorm{[\Dirac,a]}{}{}$ for all $a\in\A$ such that $a\,\dom{\Dirac}\subseteq\dom{\Dirac}$ and $[\Dirac,a]$ is bounded, where $(\A,\Hilbert,\Dirac)$ is a metric spectral triple. Assume that $\Lip(a)=\Lip(a^\ast)$ for all $a\in\ \dom{\Lip}$. Let $\B\subseteq\A$, $K > 0$, and $u$ be a unitary in $\B$ such that $\Lip(u) \leq \frac{K}{2\qdiam{\A}{\Lip} K + 1}$. Fix $\mu \in \StateSpace(\A)$. Then for all $a\in\dom{\Lip}$, we obtain
  \begin{equation*}
    \Lip(u a u^\ast) = \Lip(u (a-\mu(a)) u^\ast) \leq 2\norm{a-\mu(a)}{\A} \Lip(u) + \Lip(a) \leq (2\qdiam{\A}{\Lip} K + 1)\Lip(a) \text.
  \end{equation*}
  Thus, $\left\{ u \in \mathscr{U}(\B) : \Lip(u)\leq \frac{K}{2 \qdiam{\A}{\Lip} K + 1} \right\} \subseteq \mathscr{U}(\B|\A,\Lip,K) \text.$

  If in particular $\B$ is finite dimensional, then $\Lip$ is equivalent to the quotient norm on $\faktor{\B}{\C}$, and thus there exists $R>0$ such that $\Lip(u)\leq R$ for all unitary $u \in \B$, hence $\mathscr{U}(\B|\A,\Lip,2\qdiam{\A}{\Lip}R+1)$ is in fact the entire unitary group of $\B$.
\end{remark}

We now can put our observations together to get a convergence result for certain groups of essentially inner quantum isometries.
Indeed, the  following corollary is a consequence of Theorem \eqref{extending-auto-thm} in combination with other  results proven in this section.

\begin{corollary}
    For each $n\in\Nbar \coloneqq \N\cup\{\infty\}$, let $(\A_n,\Lip_n)$ be a {\qcms}, such that $\A_\infty = \closure{\bigcup_{n\in\N} \A_n}$, where $(\A_n)_{n\in\N}$ is an increasing (for $\subseteq$) sequence of C*-subalgebras of $\A_\infty$, with the unit of $\A_\infty$ in $\A_0$. Assume that the identity automorphism of $\A_\infty$ is also a bridge builder. 
  If we define, for all $n\in\N$,
  \begin{multline*}
    H_n \coloneqq \Big\{ \AdRep{u} \in \LargeIso(\A_\infty,\Lip_\infty) : u \in \mathscr{U}(\A_\infty), \\ \forall k\in\{0,1,\ldots,n-1\} \quad \forall a\in\dom{\Lip_k} \quad \Lip_k(u a u^\ast)=\Lip_k(a)  \Big\} \text.
  \end{multline*}
  then
  \begin{equation*}
    \lim_{n\rightarrow\infty}\covpropinquity{}\left(\left(\A_n,\Lip_n,\closure{H_n}\right),\left(\A_\infty,\Lip_\infty,\closure{\bigcup_{n\in\N}H_n}\right)\right) = 0\text,
  \end{equation*}
  where the closure $\closure{\cdot}$ is taken in the topology of pointwise convergence.
\end{corollary}

\begin{proof}
  Setting $G_\infty \coloneqq \closure{\bigcup_{n\in\N}H_n}$ and $G_n \coloneqq \closure{H_n}$ for all $n\in\N$, Theorem (\ref{extending-auto-thm}) proves that $G_n = \left\{ \alpha\in G_\infty : \alpha_{\A_n} \in \LargeIso(\A_n,\Lip_n) \right\}$ since inner automorphisms of $\A_n$ obviously extend to $\A_\infty$. Since the identity is a bridge builder, our result now follows from Corollary (\ref{id-cov-cor}).
\end{proof}

\subsection{AF Algebras with Faithful Traces}

Antonescu/Ivan and Christensen constructed in \cite{CI} a metric spectral triple on AF algebras. We now apply some of our results  to  prove the convergence of isometry groups of  their spectral triple. We begin with a description of the setup of \cite{CI}.

Let $\A = \closure{\bigcup_{n\in\N}\A_n}$ be a unital C*-algebra arising as the closure of the union of an increasing union of finite dimensional C*-subalgebras --- i.e., $\A$ is an AF algebra. We assume for convenience that $\A_0 = \C 1$ where $1$ is the unit of $\A$. Let $\varphi \in \StateSpace(\A)$ be a \emph{faithful} state of $\A$, and let $L^2(\A,\varphi)$ be the GNS Hilbert space obtained by completion of $\A$ for the inner product $a,b\in\A\mapsto\varphi(b^\ast a)$. Of course, $\A$ acts by left multiplication on $L^2(\A,\varphi)$. We write $\Omega$ for the unit $1$ of $\A$ when seen as a vector in $L^2(\A,\varphi)$, which is then a cyclic and separating vector for $\A$.

For each $n\in\N$, we identify $L^2(\A_n,\varphi)$ with the closure of $\A_n\Omega$ in $L^2(\A,\varphi)$, and we denote the orthogonal projection of $L^2(\A,\varphi)$ onto $L^2(\A_n,\varphi)$ by $P_n$. For all $a\in \A$, let $\mathds{E}_n(a)$ be defined by: $\forall a \in \A_n \quad \mathds{E}_n(a)\Omega \coloneqq P_n (a\Omega)$ --- this indeed is well-defined since $\Omega$ is separating. Since $\A_n$ is finite dimensional, it agrees, as a vector space, with $L^2(\A_n,\varphi)$, and thus $\mathds{E}_n(a) \in \A_n$. It is easy to check that $\mathds{E}_n$ thus defined is a continuous linear function. In fact, when $\varphi$ is also a trace, then $\mathds{E}_n$ is the unique conditional expectation from $\A_\infty$ onto $\A_n$ such that $\varphi\circ \mathds{E}_n = \varphi$ (see \cite[Proposition 2.6.2]{GHJ}).

\medskip

We now follow the construction of a spectral triple on AF algebras given in \cite{CI}. For each $n\in\N\setminus\{0\}$, we now let $Q_n \coloneqq P_{n} - P_{n-1}$; of course $Q_n$ is a projection since $P_n$ and $P_{n-1}$ commute and by construction, $P_n\geq P_{n-1}$. Now fix $a \in \A_n$, so $a\Omega \in L^2(\A_n,\varphi)$. By construction, $P_m (a\Omega) = a\Omega$, and thus $Q_m (a\Omega) = 0$ for all $m>n$. By continuity of $Q_m$, we thus conclude that $Q_m L^2(\A_n,\varphi) = \{ 0 \}$ for all $m>n$. This allows us to define the following operator.

For any strictly increasing sequence $\lambda\coloneqq(\lambda_n)_{n\in\N}$ of positive real numbers, we define
\begin{equation}\label{AF-Dirac-eq}
\forall \xi \in \bigcup_{n\in\N} L^2(\A_n,\varphi) \quad D_{\lambda}\xi \coloneqq \sum_{n=0}^\infty \lambda_n Q_n \text{ on }\bigcup_{n\in\N} L^2(\A_n,\varphi) \text,
\end{equation}
noting that since $(L^2(\A_n,\varphi))_{n\in\N}$ is increasing for inclusion, the set $\bigcup_{n\in\N} L^2(\A_n,\varphi)$ is a subspace of $L^2(\A,\varphi)$. Since $\bigcup_{n\in\N} L^2(\A_n,\varphi)$ is dense in $L^2(\A,\varphi)$, the operator $D_\lambda$ is defined on a dense subspace. It is obviously symmetric, by construction, and in fact, if
\begin{equation*}
T = \sum_{n=1}^\infty((\lambda_n+i)^{-1})Q_n
\end{equation*}
where the series converge in norm, then $T$ is bounded and $T (D_\lambda+i) \xi = \xi$ for all $\xi \in \bigcup_{n\in\N}L^2(\A_n,\varphi)$. Therefore, $D_\lambda$ is essentially self-adjoint. Let $\Dirac_\lambda$ be the closure of $D_\lambda$. It is then easy to see that $\Dirac_\lambda$ is a self-adjoint operator and $T$ is the inverse of $\Dirac_\lambda+i$. Since $T$ is compact, $\Dirac_\lambda$ has compact resolvent.

\medskip

Now fix $a \in \A_n$. We note that $a$ commutes with $P_k$ for $k\geq n$ by construction, since $a b \in \A_k$ for all $b\in \A_k$ with $k \geq n$. In turn, this shows that
\begin{align}
[\Dirac_\lambda,a]
&= \sum_{j=0}^\infty \lambda_j [Q_j,a] = \sum_{j=0}^n \lambda_j [Q_j,a] \label{D-infinity-D-n-eq} \\
&= P_k\left( \sum_{j=0}^\infty [Q_j,a] \right) P_k = P_k [\Dirac_\lambda,a] P_k \text.
\end{align}

From this, we see that $\{a\in\A_\infty : a\,\dom{\Dirac_\lambda}\subseteq\dom{\Dirac_\lambda}\}$ is dense in $\A_\infty$, as it contains $\bigcup_{n\in\N}\A_n$, and thus $(\A_\infty,L^2(\A_\infty,\varphi),\Dirac_\lambda)$ is a spectral triple. Moreover, so is $(\A_n,L^2(\A_n,\varphi),\Dirac_n)$ where $\Dirac_n$ is the restriction of $\Dirac_\infty$ to $L^2(\A_n,\varphi)$, which is now a bounded self-adjoint operator. Moreover, by Equation (\ref{D-infinity-D-n-eq}), we have for all $a\in\A_n$:
\begin{equation*}
\opnorm{[\Dirac_n,a]}{}{L^2(\A_n,\varphi)} = \opnorm{[\Dirac_\infty,a]}{}{L^2(\A_\infty,\varphi)} \text.
\end{equation*}

\medskip

Not all choices of a sequence $\lambda$ as above leads to a metric spectral triple. The following result addresses this matter in \cite{CI}.

\begin{theorem}[Theorem 2.1]
	There exists a strictly increasing sequence $\lambda\coloneqq (\lambda_n)_{n\in\N}$ of positive numbers, and a summable sequence $(\beta_n)_{n\in\N}$, such that 
	\begin{equation}\label{E-beta-eq}
	\forall a \in \sa{\A_\infty} \qquad \norm{ \mathds{E}_n(a) - \mathds{E}_{n+1} (a)}{\A} < \beta_n \opnorm{[\Dirac_\lambda,a]}{}{L^2(\A_\infty,\varphi)} \text.
	\end{equation}
	
	Consequently, for this choice of sequence $\lambda$, the spectral triples $\left(\A_n,L^2(\A_n,\varphi),(\Dirac_\lambda)_{|L^2(\A_n,\varphi)}\right)$ are metric for all $n\in \N\cup\{\infty\}$.
\end{theorem}

Since we have constructed metric spectral triples, we have constructed {\qcms s} and we can ask about their convergence for the propinquity. In fact, we now prove a convergence result for the \emph{spectral propinquity} $\spectralpropinquity{}$ of the spectral triples, which implies the convergence of the underlying {\qcms s}. 

Moreover, by \cite[Theorem 3.17]{FaLaPa}, to show convergence, it is sufficient to find a bridge builder which is also a full quantum isometry.

We refer to \cite{Latremoliere18g} for the definition of the spectral propinquity, and we refer to \cite[Theorem 3.17]{FaLaPa} to see that, when working in the context of AF algebras, it is sufficient for the following result to find a bridge builder which is also a full quantum isometry. 

\begin{theorem}\label{conv-AF}
	Let $\A=\closure{\bigcup_{n\in\N}\A_n}$ be a unital C*-algebra arising as  the union of an increasing sequence of finite dimensional C*-subalgebras $\A_n$ with $\A_0 = \C1$. Let $\varphi \in \StateSpace(\A)$ be a faithful state of $\A$.
	
	If $\lambda \coloneqq (\lambda_n)_{n\in\N}$ is a strictly increasing sequence of positive real numbers such that $(\A_\infty,L^2(\A_\infty,\varphi),\Dirac_\lambda)$ is a metric spectral triple, then, setting $\Dirac_n$ to be the restriction of $\Dirac_\lambda$ to $\dom{\Dirac_\lambda}\cap\Hilbert_n$, we have $(\A_n,\Hilbert_n,\Dirac_n)$ is a metric spectral triple, and
        \begin{equation*}
	\lim_{n\rightarrow\infty} \spectralpropinquity{}((\A_n,L^2(\A_n,\varphi),\Dirac_n),(\A_\infty,L^2(\A_\infty,\varphi),\Dirac_\lambda))) = 0 \text,
	\end{equation*}
	and therefore, in particular,
	\begin{equation*}
	\lim_{n\rightarrow\infty} \dpropinquity{}\left(\left(\A_n,\opnorm{[(\Dirac_n,\cdot]}{}{L^2(\A_n,\varphi)}\right), \left(\A_\infty, \opnorm{[\Dirac_\lambda,\cdot]}{}{L^2(\A_\infty,\varphi)}\right)\right) = 0 \text.
	\end{equation*}
\end{theorem}

\begin{proof}
	We will prove that the identity of $\A_\infty$ is a bridge builder. First, we note that for all $a\in\bigcup_{n\in\N}\sa{\A_n} \subseteq\dom{\Lip_\infty}$, the sequence $(\mathds{E}_n(a))_{n\in\N}$ is eventually constant equal to $a$, and thus $\lim_{n\rightarrow\infty}\mathds{E}_n(a) = a$. If $a\in\sa{\A}$, then for all $\varepsilon > 0$, there exists $a'\in\bigcup_{n\in\N}\sa{\A_n}$ such that $\norm{a-a'}{\A_\infty}<\frac{\varepsilon}{2}$; let $N' \in \N$ such that $a' \in \A_n$ for all $n\geq N'$. Then
	\begin{align*}
	\norm{a-\mathds{E}_n(a)}{\A_\infty}
	&\leq\norm{a-a'}{\A_\infty} + \norm{a'-\mathds{E}_n(a')}{\A_\infty} + \norm{\mathds{E}_n(a'-a)}{\A_\infty} \\
	&<\frac{\varepsilon}{2} + 0 + \frac{\varepsilon}{2} = \varepsilon \text.
	\end{align*}
	Thus $\lim_{n\rightarrow\infty}\norm{\mathds{E}_n(a)-a}{\A_\infty} = 0$ for all $a\in\A$.
	
	Now, fix $\varepsilon > 0$. Let $N\in\N$ such that, if $n\geq N$, then $\sum_{n\geq N}\beta_m < \varepsilon$. For all $n\in\Nbar$, let $\Lip_n$ be defined as $\Lip_{\Dirac_n}$.

	Let $a\in\dom{\Lip_\infty}$. By Equation (\ref{D-infinity-D-n-eq}), we have
	\begin{equation*}
	[\Dirac_n,\mathds{E}_n(a)] = P_n[\Dirac,a]P_n,
	\end{equation*}
	so $\Lip_n(\mathds{E}_n(a)) \leq \Lip_\infty(a)$. On the other hand,
	\begin{equation*}
	\norm{a-\mathds{E}_n(a)}{\A} \leq \sum_{k=n}^\infty \norm{\mathds{E}_{k+1}(a)-\mathds{E}_k(a)}{\A} \leq \Lip_\infty(a) \sum_{k=n}^\infty \beta_k < \varepsilon\Lip_\infty(a) \text.
	\end{equation*}
	
	Moreover, if $a\in\dom{\Lip_n}$, then we simply note that $\Lip_\infty(a) = \Lip_n(a)$ and $\norm{a-a}{\A_\infty} = 0$.
	
	We have shown that the identity is a bridge builder; of course it is also a full quantum isometry. Therefore, by \cite[Theorem 3.17]{FaLaPa}, we have the claimed convergence.
\end{proof}

Now, assume that our state $\varphi$ is in fact a tracial state. Theorem 3.1 of \cite{Conti21} gives us the following description of the group $\smallIso(\A_\infty,L^2(\A_\infty,\varphi),\Dirac_\lambda)$:
\begin{equation*}
\smallIso (\A_\infty, L^2(\A_\infty,\varphi), \Dirac_\lambda) = \{ \alpha \in \Aut{\A_\infty} :  \varphi \circ \alpha = \varphi, \forall n \in \N \quad \alpha(\A_n) = \A_n \} \text.
\end{equation*}
Let thus $G_\infty$ be any closed subgroup of $\smallIso(\A_\infty,L^2(\A_\infty,\varphi),\Dirac_\lambda)$, and if $G_n \coloneqq \{ \alpha_{|\A_n} : \alpha \in G_\infty\}$, then by Theorem (\ref{compact-Iso-thm}):
\begin{itemize}
\item $G_n$ is a compact subgroup of $\smallIso(\A_n,L^2(\A_n,\varphi),\Dirac_n)$ --- indeed, it is a compact subgroup of $\LargeIso(\A_n,\Lip_n)$, and moreover, it is immediate that if $\alpha\in G_n$, then $\alpha(\A_k)=\A_k$ for all $k\in\{0,\ldots,n\}$, so by the same observation as \cite[Theorem 3.1]{Conti21}, $\alpha\in\smallIso(\A_n,L^2(\A_n,\varphi),\Dirac_n)$;
\item we have $\lim_{n\rightarrow\infty} \covpropinquity{}((\A_n,\Lip_n,G_n),(\A_\infty,\Lip_\infty,G_\infty)) = 0 \text.$
\end{itemize}

In the special case when  $G_\infty = \smallIso(\A_\infty,L^2(\A_\infty,\varphi),\Dirac_\lambda)$, then $\smallIso(\A_\infty,L^2(\A_\infty,\varphi),\Dirac_\lambda)$ is in particular the limit, for $\Upsilon$, of the groups obtained by restriction of the elements of $G_\infty$ to $\A_n$. Since $\Upsilon$ is a metric on proper metric groups up to isometric isomorphism, we note that if $(H_n)_{n\in\N}$ is any sequence of proper metric groups converging to $\smallIso(\A_\infty,L^2(\A_\infty,\varphi),\Dirac_\lambda)$, we have $\lim_{n\rightarrow\infty}\Upsilon(H_n,G_n) = 0$. It is also noteworthy that the metric $\KantorovichDist{\Lip_n}$ on  $\smallIso(\A_\infty,L^2(\A_\infty,\varphi),\Dirac_\lambda)$ is bi-invariant for all $n\in\Nbar$, and that the restriction of $\Upsilon$ to the class of proper monoids endowed with bi-invariant metrics is complete.

\medskip

Moreover, if $\A_\infty$ is, in fact, a UHF algebra, with $\A_n = \otimes_{j=0}^n \alg{M}(\delta_j)$ where $\delta_0=1$, $(\delta_n)_{n\in\N}$ a sequence of natural numbers, and $\alg{M}(d)$ is the C*-algebra of $d\times d$ matrices, then for all $n\in\N$,
\begin{equation*}
\smallIso(\A_n,L^2(\A_n,\varphi),(\Dirac_\lambda)_{|\A_n}) = \left\{ \otimes_{j=0}^n \AdRep{u_j} : \forall j \in \{0,\ldots,d\} \quad u_j\in\alg{M}(\delta_j), u_j^\ast=u_j^{-1} \right\} \text,
\end{equation*}
and
\begin{equation*}
\smallIso(\A_\infty,L^2(\A_\infty,\varphi),\Dirac_\lambda) = \left\{ \alpha \in \Aut{\A_\infty} : \alpha_{|\A_n} \in \smallIso(\A_n,L^2(\A_n,\varphi),\Dirac_\lambda) \right\} \text.
\end{equation*}
Therefore, we immediately conclude, from Theorem (\ref{thm:iso-conv}), that
\begin{equation*}
\lim_{n\rightarrow\infty} \covpropinquity{}((\A_n,\Lip_n,\smallIso(\A_n,L^2(\A_n,\varphi),\Dirac_n), (\A_\infty,\Lip_\infty,\smallIso(\A_\infty,L^2(\A_\infty,\varphi),\Dirac_\lambda))) = 0 \text.
\end{equation*}

\subsection{Dual Actions and Inductive Limits}\label{sec:dual-actions}

We prove that the dual actions on such C*-algebras as noncommutative solenoids and Bunce-Deddens algebras are limits of dual actions on other group C*-algebras, when we endow them with the metric spectral triples introduced in \cite{FaLaPa}. Let $G_\infty$ be a discrete Abelian group, and $(G_n)_{n\in\N}$ be an increasing sequence of subgroups of $G_\infty$ such that $G_\infty = \bigcup_{n\in\N} G_n$. For each $n\in \Nbar$, let $\widehat{G_n}$ be the Pontryagin dual of $G_n$; in particular, $\widehat{G_n}$ is a compact group and $\widehat{G_\infty}$ is the projective limit of $(\widehat{G_n})_{n\in\N}$.

Let $\sigma$ be any $\T$-valued $2$-cocycle on $G_\infty$; we will denote the restriction of $\sigma$ to $G_n$ again by $\sigma$, and note that it remains a $2$-cocycle on $G_n$, for all $n\in\N$. As explained in \cite[Section 4.1]{FaLaPa}, we can identify the twisted convolution C*-algebra $C^\ast(G_n,\sigma)$ with a C*-subalgebra of $C^\ast(G_\infty,\sigma)$, with $1 \in C^\ast(G_n,\sigma)$, for all $n\in\N$, and moreover: $C^\ast(G_\infty,\sigma) = \closure{\bigcup_{n\in\N}C^\ast(G_n,\sigma)}$. We denote the dual action of $\widehat{G_n}$ on $C^\ast(G_n,\sigma)$ by $\alpha_n$ for all $n\in\Nbar$.

We fix a finite dimensional space $E$ and two unitaries $\gamma_1$ and $\gamma_2$ acting on $E$ such that $\gamma_1^2=\gamma_2^2=1$ and $\gamma_1\gamma_2 =- \gamma_2\gamma_1$. We now denote, the space of all functions $\xi : G_\infty \rightarrow E$ such that $\sum_{g \in G_\infty} \norm{\xi(g)}{E}^2 < \infty$ by $\Hilbert_\infty \coloneqq \ell^2(G_\infty,E)$, and for each $n\in\N$, we identify in the obvious way $\ell^2(G_n,E)$ with the closed subspace $\Hilbert_n \coloneqq \{ \xi \in \ell^2(G_\infty,E) : \forall g \notin G_n :\quad \xi(g) = 0 \}$ of $\ell^2(G_\infty,E)$. 

Now, let $\mathds{L}_H$ be a length function over $G_\infty$ such that $(G_n)_{n\in\N}$ converges, in the Hausdorff distance induced by $\mathds{L}_H$, to $G_\infty$. Lastly, for all $g \in G_\infty$, let $\mathds{F}(g) \coloneqq \mathrm{scale}(\min\{n \in \N : g \in G_n\})$, where $\mathrm{scale} : G_\infty \mapsto [0,\infty)$ be some strictly increasing function. We then define
\begin{equation*}
  \dom{\Dirac_\infty} \coloneqq \left\{ \xi \in \ell^2(G_\infty,E) : \sum_{g \in G_\infty} (\mathds{L}_H(g)^2 + \mathds{F}(g)^2)\norm{\xi(g)}{E}^2 < \infty \right\}
\end{equation*}
and, for all $\xi \in \ell^2(G_\infty,E)$,
\begin{equation*}
  \Dirac_\infty\xi : g \in G_\infty \mapsto \mathds{L}_H(g) \gamma_1\xi(g) + \mathds{F}(g)\gamma_2\xi(g) \text.
\end{equation*}
Moreover, for each $n\in\N$, we let $\dom{\Dirac_n} \coloneqq \dom{\Dirac_\infty}\cap\Hilbert_n$ and we let $\Dirac_n$ be the restriction of $\Dirac_\infty$ to $\dom{\Dirac_n}$.

The triples $(C^\ast(G_n,\sigma),\Hilbert_n,\Dirac_n)$ are even spectral triples for all $n\in\Nbar$ by \cite[Lemma 4.7]{FaLaPa}. In certain cases, namely when $\mathds{L} \coloneqq \mathds{L}_H + \mathds{F}$ is a length function over $G_\infty$ with the doubling property in the sense of \cite{Rieffel17}, these spectral triples are actually metric, by \cite[Lemma 4.15]{FaLaPa}, and moreover, the identity is a bridge builder for this data by \cite[Theorem 4.16]{FaLaPa}. So, in particular,
\begin{equation*}
  \lim_{n\rightarrow\infty} \spectralpropinquity{}((C^\ast(G_n,\sigma),\Hilbert_n,\Dirac_n),(C^\ast(G_\infty,\sigma),\Hilbert_\infty,\Dirac_\infty)) = 0 \text.
\end{equation*}
\medskip

Fix $n\in\Nbar$. As noted in \cite{FaLaPa}, if we define, for all $\omega \in \widehat{G_n}$, the unitary $\nu^z$ on $\ell^2(G_n,E)$ by setting for all $\xi \in \ell^2(G_n,E)$:
\begin{equation*}
  \nu_n^\omega\xi : g \in G_n \mapsto \omega(g) \xi(g) \text,
\end{equation*}
then $\nu_n^{\omega} a \nu_n^{\omega\ast} = \alpha_n^\omega(a)$ for all $a\in C^\ast(G_n,\sigma)$. Moreover, $u_n^\omega$ commutes with $\Dirac_n$, and thus for all $\omega\in \widehat{G_n}$, we have $\alpha_n^\omega \in \smallIso(C^\ast(G_n,\sigma),\Hilbert_n,\Dirac_n)$.

Let $f_n : \widehat{G_\infty} \rightarrow \widehat{G_n}$ be the surjective group homomorphism induced by duality from the inclusion $G_n\hookrightarrow G_\infty$ --- namely, $f_n$ is the restriction map from $\widehat{G_\infty}$ to $\widehat{G_n}$. By construction, $\alpha_n^{f(\omega)}$ is the restriction of $\alpha_\infty^\omega$ to $C^\ast(G_n,\sigma)$ since if $f \in C_c(G_n)$, then $\alpha_n^{f(\omega)}(f) : g \in G_n \mapsto f_n(\omega)(g) \xi(g) = \omega(g)\xi(g) = \alpha_\infty^\omega(f)(g)$, and by density of $C_c(G_n)$ in $C^\ast(G_n,\sigma)$ and continuity of the dual action. Therefore, if $\omega\in \widehat{G_\infty}$, then $\alpha_\infty^\omega$ restricts to an element of $\smallIso(C^\ast(G_n,\sigma),\Hilbert_n,\Dirac_n)$ for all $n\in\N$.

If we endow $\widehat{G_n}$ with the distance $g,g' \in \widehat{G_n} \mapsto \delta_n(g,g') \coloneqq \KantorovichDist{\Lip_n}(\alpha_n^g,\alpha_n^{g'})$, then $\delta_n$ induces a compact topology on $\widehat{G_n}$ which equals  the usual topology on $\widehat{G_n}$. For,  $\KantorovichDist{\Lip_n}$ induces the topology of point-wise convergence. Since the dual action $\alpha_n$ is strongly continuous,  
it follows that $\delta_n$ is continuous in the standard topology. Therefore, since the topology induced by $\delta_n$ is Hausdorff and coarser than the usual compact topology of  $\widehat{G_n}$, the distance $\delta_n$ metrizes the usual topology of the Pontryagin dual $\widehat{G_n}$ of $G_n$.

Therefore, we conclude from Theorem (\ref{thm:iso-conv}) that:

\begin{equation*}
  \lim_{n\rightarrow\infty} \covpropinquity{}((C^\ast(G_n,\sigma),\Lip_n,\widehat{G_n},\alpha_n),(C^\ast(G_\infty,\sigma),\Lip_\infty,\widehat{G_\infty},\alpha_\infty)) = 0 \text.
\end{equation*}

We thus obtain an example of convergence of special closed subgroups of $\smallIso(C^\ast(G_n,\sigma),\allowbreak \Hilbert_n,\Dirac_n)$.

\medskip

Particular examples of this situation are the noncommutative solenoids, where $G_\infty \coloneqq \left(\Z\left[\frac{1}{p}\right]\right)^d$ for $d\in\N\setminus\{0,1\}$, and for all $n\in\N$, we set $G_n \coloneqq \left\{ \frac{1}{p}\Z \right\}^d$, for any choice of a prime number $p$, and we can choose any $2$-cocycle of $G_\infty$. In this case, $\widehat{G_n}$ is a $2$-torus, and $\widehat{G_\infty}$ is ${\solenoid_p}^2$, where $\solenoid_p$ is the solenoid group $\left\{ (z_n)_{n\in\N} \in \T^\N : \forall n\in\N \quad z_{n+1}^p = z_n \right\}$.

Other examples include the Bunce-Deddens algebras, where $G_\infty \coloneqq \Z(p^\infty)\times \Z$ and $G_n \coloneqq \faktor{\Z}{p^n\Z}\times \Z$, again for any prime number $p$, and where $\Z(p^\infty) \coloneqq \{z\in \C : \exists n \in \N \quad z^{(p^n)} = 1 \}$ is the Pr{\"u}fer $p$-group, and $\faktor{\Z}{p^n\Z}$ is the group of $p^n$-th roots of unity in $\C$. We choose here the $2$-cocycle $\sigma:((\zeta,z),(\omega,y)) \in G_\infty^2 \mapsto \omega^z$. We refer to \cite{FaLaPa} for the details on these constructions. In this case, $\widehat{G_n}$ is the group $\faktor{\Z}{p\Z}\times\T$, and $\widehat{G_\infty}$ is the group $\Z_p\times\T$ where $\Z_p$ is the group of $p$-adic integers.

\vfill
\end{document}